\documentclass[12pt,a4paper]{amsart}

\usepackage[utf8]{inputenc}
\usepackage{graphicx}
\graphicspath{{Figures/}}
\usepackage[all]{xy}
\usepackage{listings}
\usepackage{multicol}
\usepackage{stix}
\usepackage{tikz-cd}

\newcommand{\C}{\mathbb{C}}
\newcommand{\M}{\mathcal{M}}

\newcommand{\Span}{\mathbf{Span}}
\newcommand{\RG}{\mathbf{RG}}
\newcommand{\Grpd}{\mathbf{Grpd}}
\newcommand{\Cat}{\mathbf{Cat}}
\newcommand{\MG}{\mathbf{MG}}
\newcommand{\UMG}{\mathbf{UMG}}
\newcommand{\PreGrpd}{\mathbf{PreGrpd}}
\newcommand{\Kite}{\mathbf{DiKite}}
\newcommand{\MKite}{\mathbf{MKite}}

\newcommand{\KockGrpd}{\mathbf{KockGrpd}}

\newcommand{\cod}{\mathrm{cod}}
\newcommand{\dom}{\mathrm{dom}}
\newcommand{\midop}{\mathrm{mid}}

\newtheorem{definition}{Definition}
\newtheorem{theorem}{Theorem}
\newtheorem{proposition}{Proposition}

\newtheorem{remark}{Remark}
\newtheorem{corollary}{Corollary}

\begin{document}

\date{PAB --- \today.}

\title[The Lawvere Condition]{The Lawvere Condition}

\author[N. Martins-Ferreira]{Nelson Martins-Ferreira$^{*}$}

\thanks{$(*)$ Polytechnic of Leiria, Portugal}
\address{Department of Mathematics, ESTG, Campus 2\\
Morro do Lena – Alto do Vieiro
Apartado 4163\\
2411-901 Leiria – Portugal}
\address{CDRSP, Rua de Portugal - Zona Industrial\\ 2430-028 Marinha Grande}
\email{martins.ferreira@ipleiria.pt}




\begin{abstract}
The original Lawvere condition asserts that every reflexive graph admits a unique natural structure of internal groupoid. This property was identified by P.~T.~Johnstone, following a question by A.~Carboni and a suggestion by F.~W.~Lawvere, and it plays a central role in the characterization of naturally Mal’tsev categories.
A broad and conceptually rich generalization emerges when the condition is formulated relative to a chosen class of spans. In this setting, the familiar Mal’tsev situation is recovered when the class consists of internal relations (that is, jointly monic spans) in which case the condition states that every internal reflexive relation is an equivalence relation.
The purpose of this paper is to establish a comprehensive equivalence theorem that unifies the various categorical and diagrammatic formulations of the relative Lawvere condition.
Furthermore, this formulation retains its significance even beyond the context of categories with finite limits, extending, for example, to categories admitting pullbacks of split epimorphisms along split epimorphisms. In addition to providing a fresh perspective on previously established results, we present a new characterization involving Janelidze–Pedicchio pseudogroupoids.

\end{abstract}

\maketitle

\section{Introduction}

The Lawvere condition for a category asserts that every reflexive graph in that category admits a unique natural structure of internal groupoid. This insight was first highlighted by P. T. Johnstone, following a question raised by A.~Carboni that itself traced back to a suggestion of F. W. Lawvere. It is precisely this property that characterizes naturally Mal’tsev categories, linking reflexive graphs to the presence of natural Mal’tsev operations in an elegant and conceptual way \cite{PTJ,AC}.

A particularly compelling generalization arises when one formulates the Lawvere condition relative to a chosen class \( \M \) of spans, as developed in \cite{MF38}. In this relative setting, the classical Mal’tsev case is recovered when \( \M \) consists of internal relations, that is, jointly monic spans. Other choices of \( \M \) lead to broader, yet equally meaningful, categorical landscapes, highlighting the flexibility of this framework and its capacity to capture diverse structural phenomena.

The main purpose of this paper is to establish a comprehensive equivalence theorem that unifies the various categorical and diagrammatic formulations of the relative Lawvere condition. To this end, we examine several seemingly distinct conditions---the existence of canonical pregroupoid structures on spans in \( \M \), the compatibility between spans in \( \M \) and local products, the admissibility of kite diagrams, and the existence of sections or isomorphisms for certain forgetful functors---and show that they all express the same underlying principle (see Theorem \ref{thm Main extended} of Section \ref{sec: extending main thm}).

The basic definitions are collected in Section~\ref{sec: 2}. Briefly, a \emph{local product} is obtained by taking the pullback of one split epimorphism along another, and its \emph{span part} refers to the span formed by the two canonical projections (see Definition~\ref{def: basics} item (8) of Section \ref{sec: 2}). A class of spans \( \M \) is said to be \emph{closed under the kernel pair construction} if every span built up from zig-zag triples \((x, y, z)\in D(d,c)\) of a given span \((D, d, c) \in \M\) also belongs to \( \M \) (see Definition~\ref{def: basics} items (3) and (4) of Section \ref{sec: 2}).

These results hold under mild structural assumptions on the ambient category \( \C \), namely the existence of pullbacks of split epimorphisms along split epimorphisms. Such assumptions are satisfied in many familiar categories of algebraic and geometric interest, ensuring that the theory applies in broad and natural contexts.

\begin{theorem}\label{thm Main}
Let \( \C \) be a category with pullbacks of split epimorphisms along split epimorphisms, and let \( \M \) be a class of spans in \( \C \) that is closed under the kernel pair construction and contains the span part of every local product. Then the following conditions are equivalent:
\begin{enumerate}
\item The Lawvere Condition holds in $\C$ with respect to the class $\M$;
\item The forgetful functor $\Grpd(\C,\M)\to\RG(\C,\M)$ has a section;
\item The forgetful functor $\Cat(\C,\M)\to\RG(\C,\M)$ has a section;

\item For every diagram in \( \C \) of the form
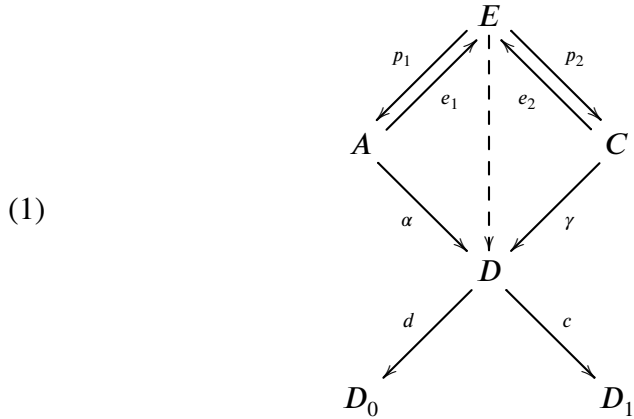
\begin{equation}\label{diag: kite main thm}
\vcenter{\xymatrix@!0@=4em{
& E \ar@{-->}[dd]
\ar@<-.5ex>[dl]_-{p_1}
\ar@<.5ex>[dr]^-{p_2}
\\
A \ar@<-.5ex>[ru]_-{e_1} \ar[rd]_-{\alpha} 
&  & 
C \ar@<.5ex>[lu]^-{e_2} \ar[ld]^-{\gamma} \\
& D \ar[dl]_{d} \ar[rd]^{c} \\
D_0 && D_1
}}
\end{equation}
if:
\begin{enumerate}
    \item the pair \( (p_1, p_2) \) is jointly monic;
    \item the spans \( (D, d, c) \) and \( (E, p_1, p_2) \) belong to \( \M \);
    \item the following identities hold:
    \begin{align*}
    & p_1 e_1 = 1_A, \quad p_2 e_2 = 1_C, \\
    & (e_1 p_1)(e_2 p_2) = (e_2 p_2)(e_1 p_1), \\
    & \alpha p_1 (e_2 p_2) = \gamma p_2 (e_1 p_1), \\
    & d \alpha p_1 = d \alpha p_1 (e_2 p_2), \quad \\ & c \gamma p_2 = c \gamma p_2 ( e_1 p_1),
    \end{align*}
\end{enumerate}
 then there exists a unique morphism \( m \colon E \to D \) such that:
\begin{align}
m e_1 = \alpha, \quad
m e_2 = \gamma, \quad
dm =d\gamma p_2, \quad cm =c\alpha p_1.    \end{align}

\item For every diagram in \( \C \) of the form
\begin{equation}\label{diag: dikite (d,c)}
\vcenter{\xymatrix@!0@=4em{
A \ar@<.5ex>[r]^-{f} \ar[rd]_-{\alpha} 
& B \ar[d]^-{\beta}
\ar@<.5ex>[l]^-{r}
\ar@<-.5ex>[r]_-{s} & 
C \ar@<-.5ex>[l]_-{g} \ar[ld]^-{\gamma} \\
& D \ar[dl]_{d} \ar[rd]^{c} \\
D_0 && D_1
}}
\end{equation}
if:
\begin{align}
& fr=1_B=gs\\
& \alpha r=\beta=\gamma s\\
& d\alpha=d\beta f\\
& c\beta g=c\gamma
\end{align}
and \[ (D,d,c)\in \M\]
 then there exists a unique morphism \[ m \colon A\times_B C \to D \] such that:
\begin{align*}
& m \langle 1_A, sf \rangle = \alpha, \quad \\
& m \langle rg,1_C \rangle = \gamma,\quad \\
& dm =d\gamma \pi_2 ,\quad \\ 
& cm =c\alpha \pi_1.    \end{align*}

\item Every span $(D,d,c)\in \M$ is equipped with a unique morphism
\begin{align*}
p_D\colon{D(d,c)\to D}
\end{align*}
 such that given any three parallel morphisms $x,y,z\colon{Z\to D}$ with $dx=dy$ and $cy=cz$ the following conditions hold \begin{align*}
dp_D\langle x,y,z\rangle=dz,\quad cp_D\langle x,y,z\rangle=cx,\\
p_D\langle x,y,y\rangle=x ,\quad p_D\langle y,y,z\rangle=z,
\end{align*}
and moreover, given any commutative diagram of the form 
\begin{align*}
\xymatrix{D_0 \ar[d]_{f_0} & D \ar[d]^{f} \ar[l]_{d} \ar[r]^{c} & D_1\ar[d]^{f_1}\\D'_0  & D' \ar[l]_{d'} \ar[r]^{c'} & D'_1}
\end{align*} 
with $(D',d',c')\in \M$, the square
\begin{align*}
\xymatrix{D(d,c)\ar[r]^-{p_D}\ar[d]_{f\times_{f_0} f\times_{f_1} f} & D\ar[d]^{f}\\
D'(d',c') \ar[r]^-{p'_{D'}} & D'}
\end{align*}
is a commutative square. 
In particular, for every object $Z$ in $\C$ and morphisms $u,v,x,y,z\colon Z\to D$ with $du=dv$, $cv=cx$, $dx=dy$, and $cy=cz$, the associative law holds:
\begin{align*}
p_D\langle u,v,p_D\langle x,y,z\rangle\rangle &= p_D\langle p_D\langle u,v,x\rangle, y,z \rangle.
\end{align*}
    
\end{enumerate}

\end{theorem}

By taking, respectively, as $\M$ the class of strong relations, relations and all spans, we get the notions of weakly Mal'tsev, Mal'tsev and naturally Mal'tsev categories. Let us recall:

\begin{definition} Let $\C$ be any  category.
\begin{enumerate}
\item \emph{Weakly Mal'tsev} \cite{MF3}\\ The category $\C$ is called \emph{weakly Mal'tsev} when it has all pullbacks of split epimorphisms along split epimorphisms, and in every local product diagram
\begin{equation}\label{local product 0.1}
\vcenter{\xymatrix{
A \ar@<-.5ex>[r]_-{e_1} & E \ar@<-.5ex>[l]_-{p_1} \ar@<.5ex>[r]^-{p_2} & C \ar@<.5ex>[l]^-{e_2}
}}
\end{equation}
the cospan $(e_1, e_2)$ is jointly epic.

\item \emph{Mal'tsev} \cite{CLP1991,Carboni1993,Fay,Mal}:\\ The category $\C$ is called \emph{Mal'tsev} when every internal reflexive relation is a tolerance relation.

 \item \emph{Naturally Mal'tsev} \cite{PTJ}:\\ The category $\C$ is called \emph{naturally Mal'tsev} if it has  binary products and every object $X$ of $\C$ is equipped with a natural Mal'tsev operation, that is, a (canonical) morphism
\[
p_X : X \times X \times X \to X
\]
such that for every morphism $f : X \to Y$ in $\C$, the square
\[
\begin{tikzcd}
X \times X \times X \arrow[r, "p_X"] \arrow[d, "f \times f \times f"'] & X \arrow[d, "f"] \\
Y \times Y \times Y \arrow[r, "p_Y"'] & Y
\end{tikzcd}
\]
commutes, and the identities
\[
p_X\langle x, y, y\rangle = x, \qquad p_X\langle y, y, z\rangle = z
\]
hold for all parallel morphisms $x, y, z : S \to X$ from any object $S$ to $X$.
\end{enumerate}
\end{definition}

A \emph{tolerance} is a relation that is both reflexive and symmetric~\cite{tol,Sossinsky1986}. There exist several equivalent ways to define a Mal'tsev category (see e.g. \cite{Bournetall}); we adopt this particular formulation because it appears to be the most suitable for our purposes, as it applies to any category. As shown in~\cite{PTJ}, within the context of finite limits, the Lawvere condition holds if and only if the category is naturally Mal'tsev.

So that there is no ambiguity, by the \emph{Lawvere condition} we mean precisely that the forgetful functor from internal groupoids to reflexive graphs is an isomorphism of categories. By the \emph{relative Lawvere condition}, with respect to a given class $\M$ of spans in a category $\C$, we mean that the forgetful functor from internal groupoids in $\C$ whose span part is in $\M$, into the category of reflexive graphs whose span part is also in $\M$, is an isomorphism of categories, which may be  represented as
$$\xymatrix{\Grpd(\C,\M)\ar[r]^{\cong} &\RG(\C,\M)}.$$ This clarification will also guide the interpretation of related conditions appearing later on. 

To streamline the presentation of items~(4) and~(5) in Theorem \ref{thm Main}, we now introduce two auxiliary definitions (see also Section~\ref{sec: 2}).

\begin{definition}\label{def: simple} Let $\C$ be any category.
\begin{enumerate}
\item \emph{Compatible spans}:\\
A span $(E,p_1,p_2)$ is said to be \emph{compatible} with a span $(D,d,c)$ if for every $\alpha$, $\gamma$, $e_1$, $e_2$, as in the diagram 
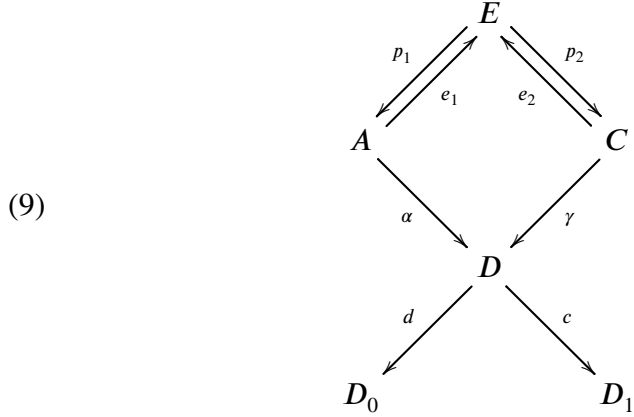
\begin{equation}\label{diag: kite main thm (def) simple}
\vcenter{\xymatrix@!0@=4em{
& E 
\ar@<-.5ex>[dl]_-{p_1}
\ar@<.5ex>[dr]^-{p_2}
\\
A \ar@<-.5ex>[ru]_-{e_1} \ar[rd]_-{\alpha} 
&  & 
C \ar@<.5ex>[lu]^-{e_2} \ar[ld]^-{\gamma} \\
& D \ar[dl]_{d} \ar[rd]^{c} \\
D_0 && D_1
}}
\end{equation}
for which the following identities hold:
    \begin{align*}
    & p_1 e_1 = 1_A, \quad p_2 e_2 = 1_C, \\
    & (e_1 p_1)(e_2 p_2) = (e_2 p_2)(e_1 p_1), \\
    & \alpha p_1 (e_2 p_2) = \gamma p_2 (e_1 p_1), \\
    & d \alpha p_1 = d \alpha p_1 (e_2 p_2), \quad \\ & c \gamma p_2 = c \gamma p_2 ( e_1 p_1),
    \end{align*}
 there exists a unique morphism \( m \colon E \to D \) such that:
\begin{align}
m e_1 = \alpha, \quad
m e_2 = \gamma, \quad
dm =d\gamma p_2, \quad cm =c\alpha p_1.    \end{align}
When a span $(E,p_1,p_2)$ is compatible with all spans in a class of spans $\M$ we say that the span $(E,p_1,p_2)$ is $\M$-compatible.

\item \emph{Dikite diagrams and their admissibility}:\\
A diagram in \( \C \) of the form
\begin{equation}\label{diag: dikite (d,c) in def simple}
\vcenter{\xymatrix@!0@=4em{
A \ar@<.5ex>[r]^-{f} \ar[rd]_-{\alpha} 
& B \ar[d]^-{\beta}
\ar@<.5ex>[l]^-{r}
\ar@<-.5ex>[r]_-{s} & 
C \ar@<-.5ex>[l]_-{g} \ar[ld]^-{\gamma} \\
& D \ar[dl]_{d} \ar[rd]^{c} \\
D_0 && D_1
}}
\end{equation}
with
\begin{align}
& fr = 1_B = gs, \\
& \alpha r = \beta = \gamma s, \\
& d\alpha = d\beta f, \\
& c\beta g = c\gamma,
\end{align}
is called a \emph{dikite diagram} provided that the pullback of $f$ and $g$ exists. Such a diagram is said to be \emph{admissible} if there exists a unique morphism
\[
m \colon A \times_B C \to D
\]
making the following equalities hold:
\begin{align*}
& m \langle 1_A, sf \rangle = \alpha, \\
& m \langle rg, 1_C \rangle = \gamma, \\
& d m = d \gamma \pi_2, \\
& c m = c \alpha \pi_1.
\end{align*}
\end{enumerate}
\end{definition}

In this light, condition~(4) of Theorem \ref{thm Main} can be restated as the assertion that every jointly monic span \((E, p_1, p_2) \in \M\) is \(\M\)-compatible. Furthermore, condition~(5) states that every dikite diagram whose direction span lies in~\(\M\) is admissible. Finally, condition~(6) establishes that the forgetful functor from (associative) pregroupoids~\cite{K1,K2,Johnstone1991} whose span part lies in~\(\M\) to the category of spans belonging to~\(\M\) is an isomorphism of categories.

Each one of the three classes of spans mentioned above contains local products and is closed under the kernel pair construction (see Definition~\ref{def: basics}). Consequently, they all satisfy the hypotheses of Theorem \ref{thm Main}.

The following three corollaries illustrate the significance of our main theorem in unifying the classes of weakly Mal’tsev, Mal’tsev, and naturally Mal’tsev categories.

Note that, rather than requiring the category $\C$ to have all kernel pairs, we restrict our analysis to spans whose legs admit kernel pairs, as further explained in Definition~\ref{def: basics}, item~(1).

The case where $\M$ is the class of all strong relations (jointly strongly monic spans---whose legs admit kernel pairs):

\begin{corollary}\label{cor:weakly Mal'tsev}
Let \( \C \) be a category with pullbacks of split epimorphisms along split epimorphisms and equalizers. Then the following conditions are equivalent:
\begin{enumerate}
\item The category $\C$ is weakly Mal'tsev;
    \item Every reflexive strong relation is an equivalence relation;
        \item Every reflexive strong relation is a preorder;

\item Every strong relation in $\C$ is compatible with the class of all strong relations;

\item Every dikite diagram whose direction is a strong relation is admissible;

  \item Every strong relation is difunctional.

\end{enumerate}

\end{corollary}

The case where $\M$ is the class of all relations (jointly monic spans---whose legs admit kernel pairs):

\begin{corollary}\label{cor:Mal'tsev}
Let \( \C \) be a category with pullbacks of split epimorphisms along split epimorphisms. Then the following conditions are equivalent:
\begin{enumerate}
\item The category $\C$ is Mal'tsev;
    \item Every reflexive relation is an equivalence relation;
        \item Every reflexive relation is a preorder;
\item Every relation is compatible with the class of all relations;
\item Every dikite diagram whose direction is a relation is admissible;
  \item Every relation is difunctional.

\end{enumerate}

\end{corollary}

The case where $\M$ is the class of all spans (whose legs admit kernel pairs):

\begin{corollary}\label{cor:natMal'tsev}
Let \( \C \) be a category with pullbacks of split epimorphisms along split epimorphisms and  finite products. Then the following conditions are equivalent:
\begin{enumerate}
\item The category $\C$ is naturally Mal'tsev;
\item The forgetful functor $\Grpd(\C)\to\RG(\C)$ has a section;
\item The forgetful functor $\Cat(\C)\to\RG(\C)$ has a section;

\item Every relation is compatible with the class of all spans;
\item Every dikite diagram is admissible;
  \item Every span has a canonical pregroupoid structure.

\end{enumerate}

\end{corollary}

At this point, it is clear that conditions~(4) and~(5) in Corollary~\ref{cor:natMal'tsev} assert that every local product is also a local coproduct (see \cite{Bourn1991}). 
Similarly, in Corollary~\ref{cor:Mal'tsev}, these conditions can be interpreted---now assuming that all pullbacks exist---as stating that the cospan $(e_1, e_2)$ arising from every local product is jointly extremally epimorphic, and that every comparison morphism in a commutative split square is a strong epimorphism (see, for example, \cite{Bournetall}).

These observations are further extended in Section~\ref{sec: extending main thm}, where they are connected with commutators and Janelidze–Pedicchio pseudogroupoids~\cite{JP,JP2,Janelidze2016}.

\section{Basic definitions of internal categorical structures}\label{sec: 2}

In the spirit of Mac\,Lane’s reflections on the “health of mathematics” \cite{MacLane1983}, the pursuit of abstraction within category theory remains vital precisely when it ventures beyond the comfort of well-behaved environments. The investigation of internal categorical structures in categories lacking all finite limits embodies this spirit, seeking to preserve conceptual coherence while relaxing structural assumptions. 

With this motivation in mind, we recall the basic definitions of internal categorical structures, since in our setting we do not assume that all limits exist. For instance, rather than working in categories where kernel pairs are guaranteed, we restrict attention to those spans whose legs possess kernel pairs. In certain contexts, we shall also require the existence of a corresponding box construction for a span; whenever this is the case, it will be explicitly indicated. 

Most of the notions involved are standard and were already well established by the 1980s. By that time, it was clear that they could be formulated in a general categorical setting; however, most results continued to be stated under the assumption of finite limits. Whenever appropriate, a reference is provided for each basic concept being defined. For a general overview with historical notes, see~\cite{Janelidze2016}.

\begin{definition}\label{def: basics} Let \(\C\) be a category with pullbacks of split epimorphisms along split epimorphisms.
\begin{enumerate}
\item  \emph{Span} \cite{Benabou1973,Kell,K2,MacLane1998,Yoneda}:\\
A diagram in $\C$ of the form
\[
\xymatrix{
& D \ar[ld]_d \ar[rd]^c \\
D_0 && D_1
}
\]
is called a \emph{span} when both morphisms $d$ and $c$ admit kernel pairs, that is, when there exist canonical pullback squares:
\[
\xymatrix{
& D(d)\ar[dl]_{d_1} \ar[rd]^{d_2} & & D(c) \ar[dl]_{c_1} \ar[rd]^{c_2}\\
D\ar[rd]^{d} & & D \ar[ld]_d \ar[rd]^c  && D\ar[ld]_{c}\\
& D_0 && D_1.
}
\]

If needed, a span is represented as a tuple $(D,d,c)$. Note that this is not the standard definition of a span, but it becomes the standard one as soon as the category $\C$ has kernel pairs.

\item \emph{The box construction} \cite{JP,JP2}:\\
In some occasions we will need the \emph{box construction} for a given span. A span $(D,d,c)$ admits the \emph{box construction} when there exists an object, denoted by $D(\square)$, together with four projection morphisms $$q_1,q_2,q_3,q_4\colon{D(\square)\to D},$$
that constitute a limiting cone in the following diagram
\[
\xymatrix{
D(d) \ar@<0ex>[r]^-{d_2} \ar@<0ex>[d]_-{d_1}
& D & D(c) \ar[l]_{c_1}\ar[d]^{c_2} 
 \\
D & D(\square)
 \ar@<0ex>[l]^-{q_1} \ar@<-.5ex>[u]_-{q_2} \ar[r]^{q_3} \ar[d]_{q_4}
& D \\
D(c) \ar[u]^{c_2}\ar[r]^{c_1}
& D & D(d) \ar[u]_{d_1}\ar[l]_{d_2}.
}
\]

Intuitively, in terms of generalized elements, we may think of  $D(\square)$ as consisting of tuples $\langle x,y,z,w\rangle$ of parallel morphisms from some object $Z$ into $D$ with $dx=dy$, $cy=cz$, $dz=dw$ and $cw=cx$  as illustrated in the case were we can take $Z=1$ and identify $x,y,z,w$ with elements in $D$ that may be seen as arrows:
\[
\xymatrix{\cdot \ar[r]^{y}\ar[d]_{x}& \cdot \\ \cdot & \cdot \ar@{-->}[l]_{w} \ar[u]_{z}}
\]
The reason why $w$ is dashed is because it will have a different treatment than the triple $\langle x,y,z\rangle$ to which we will refer to as the kernel pair construction. Moreover, we will also need the requirement that the canonical projection $\langle x,y,z,w\rangle\mapsto \langle x,y,z\rangle$ admits a kernel pair, since this construction will be needed later on when defining internal pseudogroupoid \cite{JP,JP2,Janelidze2016}.

\item \emph{The kernel pair construction} \cite{Johnstone1991}:\\ 
For a given span \((D, d, c)\), since the kernel pairs of \(d\) and \(c\) exist,  we may from the  \emph{kernel pair construction} by taking pullbacks as illustrated:
\[
\xymatrix@!0@=4em{
D(d,c) \ar@<.5ex>[r]^-{p_2} \ar@<-.5ex>[d]_-{p_1}
& D(c) \ar@<.5ex>[l]^-{e_2} \ar@<-.5ex>[d]_-{c_1} \ar@<0ex>[r]^{c_2}
& D \ar[d]_{c}
 \\
D(d)
 \ar@<.5ex>[r]^-{d_2} \ar@<-.5ex>[u]_-{e_1} \ar[d]_{d_1}
& D
 \ar@<.5ex>[l]^-{\Delta} \ar@<-.5ex>[u]_-{\Delta} \ar[r]^{c} \ar[d]_{d}
& D_1 \\
D \ar[r]^{d}
& D_0.
}
\]

In the event that the category $\C$ has all pullbacks and a terminal object (and hence all finite limits), the kernel pair construction is nothing but the familiar pullback of $\langle d,c\rangle$ along $d\times c$ as illustrated
\begin{align*}
\xymatrix{D(d,c)\ar[r]^-{\langle\dom,\cod\rangle}\ar[d]_{\midop} & D\times D\ar[d]^{d\times c}\\D\ar[r]^-{\langle d,c\rangle}& D_0\times D_1}
\end{align*}
in which case $\midop$ is obtained either by $d_2p_1$ or by $c_1p_2$, and there is an induced span \((D(d,c), \dom, \cod)\) with   \(\dom = d_1 p_1\) and \(\cod = c_2 p_2\). Note that $\dom$ and $\cod$ are both split by the diagonal morphism $\Delta\colon{D\to D(d,c)}$ and hence admit kernel pairs.

Note also that the box construction is always possible in the presence of finite limits by observing that the square
\[
\xymatrix{D(\square) \ar[d]_{q_4}\ar[r]^{\langle q_1,q_2,q_3\rangle} & D(d,c) \ar[d]^{\langle d\cod,c\dom \rangle} \\ D\ar[r]^-{\langle d,c\rangle}& D_0\times D_1}
\]
is a pullback square.

\item \emph{Closure under the kernel pair construction}:\\
We will say that a class of spans $\M$ is closed under the kernel pair construction when for every span $(D,d,c)\in \M$, the induced span $(D(d,c),\dom,\cod)$ also is in $\M$.

\item \emph{Split span}:\\
A diagram in $\C$ of the form
\begin{equation}\label{split span}
\vcenter{\xymatrix{
A \ar@<-.5ex>[r]_-{e_1} & E \ar@<-.5ex>[l]_-{p_1} \ar@<.5ex>[r]^-{p_2} & C \ar@<.5ex>[l]^-{e_2}
}}
\end{equation}
is called a split span when $p_1e_1=1_A$ and $p_2e_2=1_C$. If needed it will be referred to as a tuple $(E,e_1,p_1,e_2,p_2)$.

\item \emph{Commutative split span}:\\
A split span $(E,e_1,p_1,e_2,p_2)$ is called a \emph{commutative split span} when $$e_1p_1e_2p_2=e_2p_2e_1p_1.$$

\item \emph{Commutative split square}:\\
A diagram in $\C$ of the form
\begin{equation}\label{split square}
\xymatrix@!0@=4em{
E \ar@<.5ex>[r]^-{p_2} \ar@<-.5ex>[d]_-{p_1} & C \ar@<.5ex>[l]^-{e_2} \ar@<-.5ex>[d]_-{g} \\
A \ar@<.5ex>[r]^-{f} \ar@<-.5ex>[u]_-{e_1} & B \ar@<.5ex>[l]^-{r} \ar@<-.5ex>[u]_-{s}.
}
\end{equation}
is called a \emph{commutative split square} when $p_1e_1=1_A$, $p_2e_2=1_C$, $fr=1_B=gs$, $gp_2=fp_1$ and $e_1 r=e_2 s$. 
 
\item \emph{Local product} \cite{MF7}:\\
A diagram in $\C$ of the form
\begin{equation}\label{split span as local prod}
\vcenter{\xymatrix{
A \ar@<-.5ex>[r]_-{e_1} & E \ar@<-.5ex>[l]_-{p_1} \ar@<.5ex>[r]^-{p_2} & C \ar@<.5ex>[l]^-{e_2}
}}
\end{equation}
is called a \emph{local product} when there exists a second diagram
\begin{equation}\label{local span}
\vcenter{\xymatrix@!0@=4em{
A \ar@<.5ex>[r]^-{f} & B \ar@<.5ex>[l]^-{r} \ar@<-.5ex>[r]_-{s} & C \ar@<-.5ex>[l]_-{g}
}}
\end{equation}
with \( f r = 1_B = g s \), such that the span \( (E, p_1, p_2) \) is the pullback of \( g \) along \( f \), and the morphisms 
\( e_1 = \langle 1_A, s f \rangle \) and \( e_2 = \langle r g, 1_C \rangle \) are the canonical morphisms induced by the universal property of the pullback. In particular, it forms a commutative split span $(E,e_1,p_1,e_2,p_2)$. Conversely, if in a commutative split span $(E,e_1,p_1,e_2,p_2)$ the pushout of $p_1$ and $p_2$ exists and the span $(E,p_1,p_2)$ is the pullback of its pushout then it is a local product.

\item  \emph{Reflexive graph}:\\
A diagram in $\C$ of the form 
\[
\xymatrix{
 C_1 \ar@<1ex>[r]^d \ar@<-1ex>[r]_c & C_0 \ar[l]|e
}
\]
is called a \emph{reflexive graph} when $de=1_{C_0}=ce$. It may also be represented as a tuple $(C_1,C_0,d,e,c)$.

\item  \emph{Unital multiplicative graph} \cite{Janelidze2003,JP,JP2}:\\
 A diagram in $\C$ of the form
\[
\xymatrix{
C_2 \ar[r]^m & C_1 \ar@<1ex>[r]^d \ar@<-1ex>[r]_c & C_0 \ar[l]|e
}
\]
is called a \emph{unital multiplicative graph} when the tuple $(C_1,C_0,d,e,c)$ is a reflexive graph, and there exists a (canonical) pair of morphisms $(\pi_1,\pi_2)$ such that the square
 \[
  \xymatrix{
  C_2 \ar[r]^{\pi_2} \ar[d]_{\pi_1} & C_1 \ar[d]^c \\
  C_1 \ar[r]^d & C_0
  }
  \]
is a pullback square (that is $C_2=C_1\times_{C_0} C_1$), and moreover:
\begin{align*}
  d\,m &= d\,\pi_2, \\
  c\,m &= c\,\pi_1, \\
  m\langle 1_{C_1}, ed \rangle &= 1_{C_1} = m\langle ec, 1_{C_1} \rangle.
\end{align*}
In some cases, a unital multiplicative graph may be written as a tuple $(C_1,C_0,d,e,c,m)$, with the understanding that $m$ is the multiplication, that is, a morphism from $C_1\times_{C_0} C_1$ to $C_1$.

\item \emph{Internal category} \cite{MacLane1998}\\
 A diagram in $\C$ of the form
\[
\xymatrix{
C_2 \ar[r]^m & C_1 \ar@<1ex>[r]^d \ar@<-1ex>[r]_c & C_0 \ar[l]|e
}
\]
is called an \emph{internal category} when the tuple $(C_1,C_0,d,e,c,m)$ is a unital multiplicative graph and the multiplication $m\colon{C_2\to C_1}$ is associative, that is,
\[
m(m \times_{C_0} 1_{C_1}) = m(1_{C_1} \times_{C_0} m).
\]

\item \emph{Internal groupoid} \cite{MacLane1998}\\
 A diagram in \(\C\) of the form
\[
\xymatrix{
C_2 \ar[r]^m & C_1 \ar@(ur,ul)_i \ar@<1ex>[r]^d \ar@<-1ex>[r]_c & C_0 \ar[l]|e
}
\]
is called an \emph{internal groupoid} when the tuple $(C_1,C_0,d,e,c,m)$ is an internal category, and moreover the morphism \(i\colon C_1 \to C_1\), called the \emph{inverse}, satisfies:
\begin{align*}
m\langle 1_{C_1}, i \rangle &= ed, \\
m\langle i, 1_{C_1} \rangle &= ec.
\end{align*}
If needed, an internal groupoid such as the one displayed in the diagram above will be referred to as the tuple $(C_1,C_0,d,e,c,m,i)$.

\item \emph{Internal Pregroupoid} \cite{Johnstone1991}:\\
A diagram in $\C$ of the form
\[
\xymatrix{
& D(d,c) \ar[d]^{p}\\
 & D \ar[ld]_d \ar[rd]^c \\
D_0 && D_1
}
\]
is called an internal \emph{pregroupoid} when $(D,d,c)$ is a span, the object $D(d,c)$ is obtained via the kernel pair construction
\[
\xymatrix@!0@=4em{
D(d,c) \ar@<.5ex>[r]^-{p_2} \ar@<-.5ex>[d]_-{p_1}
& D(c) \ar@<.5ex>[l]^-{e_2} \ar@<-.5ex>[d]_-{c_1} \ar@<0ex>[r]^{c_2}
& D \ar[d]_{c}
 \\
D(d)
 \ar@<.5ex>[r]^-{d_2} \ar@<-.5ex>[u]_-{e_1} \ar[d]_{d_1}
& D
 \ar@<.5ex>[l]^-{\Delta} \ar@<-.5ex>[u]_-{\Delta} \ar[r]^{c} \ar[d]_{d}
& D_1 \\
D \ar[r]^{d}
& D_0,
}
\]
and the morphism $p$ satisfies
\begin{align}
p\,e_1 &= d_1, \quad p\,e_2 = c_2, \label{Mal'tsev-conditions} \\
d\,p &= d\,c_2\,p_2, \quad c\,p = c\,d_1\,p_1. \label{Domain-and-Codomain}
\end{align}
Intuitively, \(D(d,c)\) consists of triples of morphisms \(x,y,z\colon Z\to D\) such that \(dx = dy\) and \(cy = cz\); under this interpretation, the four conditions on \(p\) become
\begin{align*}
p\langle x,y,y\rangle &= x, \quad  p\langle y,y,z\rangle = z, \\
dp\langle x,y,z\rangle &= dz, \quad  cp\langle x,y,z\rangle = cx.
\end{align*}
If needed, a pregroupoid may be represented as a tuple $(D,d,c,p)$.

\item \emph{Kock pregroupoid} \cite{K1,K2}:\\
A diagram in $\C$ of the form
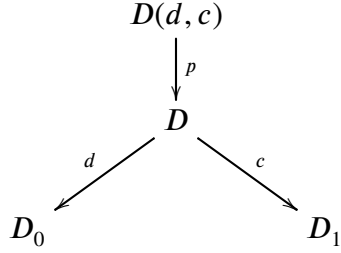
\[
\xymatrix{
& D(d,c) \ar[d]^{p}\\
 & D \ar[ld]_d \ar[rd]^c \\
D_0 && D_1
}
\]
is called a \emph{Kock pregroupoid} (or \emph{associative pregroupoid}) when the tuple $(D,d,c,p)$ is an internal pregroupoid, and the morphism $p$ satisfies the following condition: for every object $Z$ in $\C$ and morphisms $u,v,x,y,z\colon Z\to D$ with $du=dv$, $cv=cx$, $dx=dy$, and $cy=cz$, the associative law holds:
\begin{align*}
p\langle u,v,p\langle x,y,z\rangle\rangle &= p\langle p\langle u,v,x\rangle, y,z \rangle.
\end{align*}

\item \emph{Autonomous pregroupoid}:\\
An associative pregroupoid $(D,d,c,p)$ is called \emph{autonomous} when for every object $Z$ in $\C$ and for any given nine parallel morphisms    $ x_1,y_1,z_1, x_2,y_2,z_2,x_3,y_3,z_3\colon{  Z\to D
}$
    with
    \begin{align*}
    dx_1=&dx_2  &  dx_1=&dy_1\\
    &cx_2=cx_3 && cy_1=cz_1\\
    dy_1=&dy_2  &  dx_2=&dy_2\\
    &cy_2=cy_3 && cy_2=cz_2\\
    dz_1=&dz_2  &  dx_3=&dy_3\\
    &cz_2=cz_3 && cy_3=cz_3
    \end{align*}
  as ilustrated
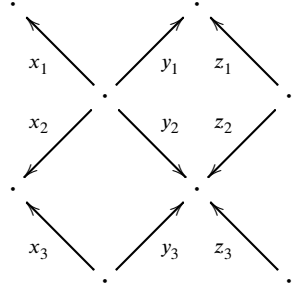
\[\xymatrix{
\cdot && \cdot \\
& \cdot 
\ar[lu]^{x_1} \ar[ru]_{y_1}
\ar[ld]_{x_2} \ar[rd]^{y_2}
&&\cdot \ar[lu]^{z_1} \ar[ld]_{z_2} \\
\cdot && \cdot \\
& \cdot 
\ar[lu]^{x_3} \ar[ru]_{y_3}
&&\cdot \ar[lu]^{z_3}
}\]
   it follows that 
   \begin{align*}
      p\langle p\langle x_1,x_2,x_3\rangle,\;p\langle y_1,y_2,y_3\rangle,\;p\langle z_1,z_2,z_3\rangle\rangle=\\
      =p\langle p\langle x_1,y_1,z_1\rangle,\;p\langle x_2,y_2,z_2\rangle,\;p\langle x_3,y_3,z_3\rangle\rangle.
   \end{align*}

\item \emph{Janelidze-Peddicchio pseudogroupoid} \cite{JP,JP2,Janelidze2016}:\\
A diagram in $\C$ of the form
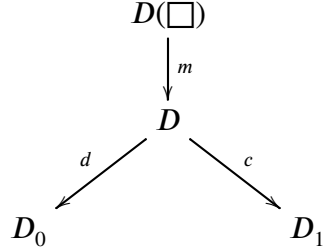
\[
\xymatrix{
& D(\square) \ar[d]^{m}\\
 & D \ar[ld]_d \ar[rd]^c \\
D_0 && D_1
}
\]
is called an internal \emph{pseudogroupoid} when $(D,d,c)$ is a span that admits the box construction (see above), thus given rise to the object $D(\square)$ and the four projections $q_1,q_2,q_3,q_4\colon{D(\square)\to D}$ 
and moreover, the morphism $m$ is required to satisfy the following condition: for every object $Z$ and every five parallel morphisms $$x,y,z,w,w'\colon{Z\to D}$$ with $dx=dy$, $cy=cz$, $dz=dw$, $cw=cx$ and $dz=dw'$, $cw'=cx$
\begin{align}
m\langle x,y,y,x\rangle &= x, \quad m\langle y,y,z,z\rangle = z, \label{Mal'tsev-conditions} \\
dm\langle x,y,z,w\rangle &= dz, \quad cm\langle x,y,z,w\rangle = cx. \label{Domain-and-Codomain}\\
m\langle x,y,z,w\rangle &= m\langle x,y,z,w'\rangle 
\end{align}
 as illustrated when the object $Z$ can be taken as the terminal object and the generalized elements in $D$ can be pictured as arrows.
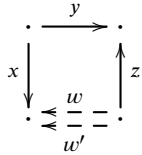
\[
\xymatrix{\cdot \ar[r]^{y}\ar[d]_{x}& \cdot \\ \cdot & \cdot \ar@<.5ex>@{-->}[l]^{w'} \ar@<-.5ex>@{-->}[l]_{w} \ar[u]_{z}}
\]
Note that in comparison to \cite{JP,JP2,Janelidze2016} we have chosen to  project the fourth variable $w$ instead of the third one $z$, and we are not considering the associative condition.

\item \emph{Directed kite} \cite{MF17}\\
A diagram in $\C$ of the form
\begin{equation}\label{diag: dikite (d,c) in definition}
\vcenter{\xymatrix@!0@=4em{
A \ar@<.5ex>[r]^-{f} \ar[rd]_-{\alpha} 
& B \ar[d]^-{\beta}
\ar@<.5ex>[l]^-{r}
\ar@<-.5ex>[r]_-{s} & 
C \ar@<-.5ex>[l]_-{g} \ar[ld]^-{\gamma} \\
& D \ar[dl]_{d} \ar[rd]^{c} \\
D_0 && D_1
}}
\end{equation}
is called a \emph{directed kite} (or \emph{dikite diagram}) when the following conditions hold:
\begin{align}
& fr = 1_B = gs, \\
& \alpha r = \beta = \gamma s, \\
& d\alpha = d\beta f, \\
& c\beta g = c\gamma,
\end{align}
and $(D,d,c)$ is a span, called the \emph{direction} of the dikite diagram.  

\item  \emph{Multiplicative dikite} \cite{MF17}\\
A \emph{multiplication} on a dikite such as the one displayed above is a morphism $m\colon A\times_B C \to D$ satisfying:
\begin{align*}
& m\langle 1_A, sf \rangle = \alpha, \\
& m\langle rg, 1_C \rangle = \gamma, \\
& dm = d\gamma \pi_2, \\
& cm = c\alpha \pi_1.
\end{align*}
When such a multiplication exists and is unique, we say that the dikite diagram is \emph{admissible}.

\item 
A diagram in $\C$ of the form
\begin{align*}
\xymatrix{D_0 \ar[d]_{f_0} & D \ar[d]^{f} \ar[l]_{d} \ar[r]^{c} & D_1\ar[d]^{f_1}\\D'_0  & D' \ar[l]_{d'} \ar[r]^{c'} & D'_1}
\end{align*} 
in which $(D,d,c)$ and $(D',d',c')$ are spans is called a \emph{morphism of spans} when it is a commutative diagram.

\item 
A diagram in $\C$ of the form
\begin{align*}
\xymatrix{D_0 \ar[d]_{f_0} & D \ar[d]^{f} \ar[l]_{d} \ar[r]^{c} & D_1\ar[d]^{f_1}\\D'_0  & D' \ar[l]_{d'} \ar[r]^{c'} & D'_1}
\end{align*} 
in which $(D,d,c,p)$ and $(D',d',c',p')$ are pregroupoids is called a \emph{morphism of pregroupoids} when it is a commutative diagram and moreover the induced square
\begin{align*}
\xymatrix{D(d,c)\ar[r]^-{p}\ar[d]_{f\times_{f_0} f\times_{f_1} f} & D\ar[d]^{f}\\
D'(d',c') \ar[r]^-{p'} & D'}
\end{align*}
is a commutative square. 

Morphisms between reflexive graphs, unital multiplicative graphs, internal categories, internal groupoids, Kock pregroupoids, directed kites, and multiplicative dikites are obtained in a similar way.

\end{enumerate}
\end{definition}

The following proposition further illustrates the utility of the notion of directed kite and will play a key role in the proof of the main theorem.

\begin{proposition}\label{prop: particular cases} 
Let $\C$ be a category with pullbacks of split epimorphisms along split epimorphisms.
\begin{enumerate}
\item A tuple $(C_1,C_0,d,e,c)$ is a reflexive graph if and only if the diagram 
\begin{equation}\label{diag: kite1}
\vcenter{\xymatrix{
C_1 \ar@<.5ex>[r]^-{d} \ar@{=}[rd]_-{} & C_0
\ar@<.5ex>[l]^-{e}
\ar@<-.5ex>[r]_-{e}
\ar[d]^-{e} & C_1 \ar@<-.5ex>[l]_-{c} \ar@{=}[ld]^-{}\\
& C_1\ar[dl]_{d}\ar[rd]^{c}\\
C_0 && C_0
}}
\end{equation}
is a directed kite.

\item A tuple $(C_1,C_0,d,e,c,m)$ is a unital multiplicative graph if and only if the morphism 
\[
m\colon C_1\times_{C_0} C_1 \to C_1
\] 
defines a multiplication for the directed kite
\begin{equation}\label{diag: kite1.1}
\vcenter{\xymatrix{
C_1 \ar@<.5ex>[r]^-{d} \ar@{=}[rd]_-{} & C_0
\ar@<.5ex>[l]^-{e}
\ar@<-.5ex>[r]_-{e}
\ar[d]^-{e} & C_1 \ar@<-.5ex>[l]_-{c} \ar@{=}[ld]^-{}\\
& C_1\ar[dl]_{d}\ar[rd]^{c}\\C_0 && C_0
}}
\end{equation}

\item A diagram of the form
\begin{equation}\label{diag: morphims of reflexive graphs}
\xymatrix{
C_1 \ar@<1ex>[r]^{d} \ar@<-1ex>[r]_{c}\ar[d]_{f_1} & C_0 \ar[l]|{e}\ar[d]^{f_0} \\
C'_1 \ar@<1ex>[r]^{d'} \ar@<-1ex>[r]_{c'} & C'_0 \ar[l]|{e'} 
}
\end{equation}
represents a morphism of reflexive graphs if and only if the diagram
\begin{equation}\label{diag: kite1.2}
\vcenter{
\xymatrix{
C_1 \ar@<.5ex>[r]^-{d} \ar@{->}[rd]_-{f_1} & C_0
\ar@<.5ex>[l]^-{e}
\ar@<-.5ex>[r]_-{e}
\ar[d]|-{e'f_0} & C_1 \ar@<-.5ex>[l]_-{c} \ar@{->}[ld]^-{f_1}\\
& C'_1\ar[dl]_{d'}\ar[rd]^{c'}\\
C'_0 && C'_0
}}
\end{equation}
is a directed kite and satisfies $d'e'=1_{C'_1}=c'e'$.
Moreover, if the directed kite is admissible (that is, there exists a unique multiplication $\mu\colon C_2 \to C'_1$), then for any two multiplications on the reflexive graphs, say $m\colon C_2\to C_1$ and $m'\colon C'_2\to C'_1$, it holds that
\[
m'(f_1\times_{f_0} f_1) = f_1 \, m,
\]
so that the triple $(f_1\times_{f_0} f_1, f_1, f_0)$ defines a morphism of unital multiplicative graphs, as illustrated by the diagram
\begin{equation}\label{diag: morphims of unital multiplicative reflexive graphs}
\xymatrix{
C_2 \ar[r]^{m}\ar[d]_{f_2} & C_1 \ar@<1ex>[r]^{d} \ar@<-1ex>[r]_{c}\ar[d]_{f_1} & C_0 \ar[l]|{e}\ar[d]^{f_0} \\
C'_2 \ar[r]^{m'} & C'_1 \ar@<1ex>[r]^{d'} \ar@<-1ex>[r]_{c'} & C'_0 \ar[l]|{e'}
}
\end{equation}

\item If $(C_1,C_0,d,e,c,m)$ is a unital multiplicative graph then the diagrams 
\begin{equation}\label{diag: kite2.1}
\vcenter{\xymatrix{C_2 \ar@<.5ex>[r]^-{\pi_2} \ar[rd]_-{m} & C_1
\ar@<.5ex>[l]^-{e_2}
\ar@<-.5ex>[r]_-{e_1}
\ar@{=}[d]^-{} & C_2 \ar@<-.5ex>[l]_-{\pi_1} \ar[ld]^-{m}\\
& C_1\ar[dl]_{d}\ar[rd]^{c}\\\cdot&&\cdot}}
\end{equation}
 and 
 \begin{equation}\label{diag: kite3.1}
\vcenter{\xymatrix{C_2 \ar@<.5ex>[r]^-{m} \ar[rd]_-{\pi_2} & C_1
\ar@<.5ex>[l]^-{e_2}
\ar@<-.5ex>[r]_-{e_1}
\ar@{=}[d]^-{} & C_2 \ar@<-.5ex>[l]_-{m} \ar[ld]^-{\pi_1}\\
& C_1\ar[dl]_{d}\ar[rd]^{c}\\\cdot&&\cdot}}
\end{equation}
are directed kites.

\item A unital multiplicative graph $(C_1,C_0,d,e,c,m)$ is an internal category if and only if the directed kite
\begin{equation}\label{diag: kite2}
\vcenter{\xymatrix{C_2 \ar@<.5ex>[r]^-{\pi_2} \ar[rd]_-{m} & C_1
\ar@<.5ex>[l]^-{e_2}
\ar@<-.5ex>[r]_-{e_1}
\ar@{=}[d]^-{} & C_2 \ar@<-.5ex>[l]_-{\pi_1} \ar[ld]^-{m}\\
& C_1\ar[dl]_{d}\ar[rd]^{c}\\\cdot&&\cdot}}
\end{equation}
is admissible.

\item An internal category $(C_1,C_0,d,e,c,m)$ is  an internal groupoid if and only if the directed kite
\begin{equation}\label{diag: kite3}
\vcenter{\xymatrix{C_2 \ar@<.5ex>[r]^-{m} \ar[rd]_-{\pi_2} & C_1
\ar@<.5ex>[l]^-{e_2}
\ar@<-.5ex>[r]_-{e_1}
\ar@{=}[d]^-{} & C_2 \ar@<-.5ex>[l]_-{m} \ar[ld]^-{\pi_1}\\
& C_1\ar[dl]_{d}\ar[rd]^{c}\\\cdot&&\cdot}}
\end{equation}
is admissible.

\item If $(D,d,c)$ is a span then the kernel pair construction gives a directed kite as follows
\begin{equation}\label{diag: kite5}
\vcenter{\xymatrix{D(d) \ar@<.5ex>[r]^-{d_2} \ar@{->}[rd]_-{d_1} & D
\ar@<.5ex>[l]^-{\Delta}
\ar@<-.5ex>[r]_-{\Delta}
\ar@{=}[d]^-{} & D(c) \ar@<-.5ex>[l]_-{c_1} \ar@{->}[ld]^-{c_2}\\
& D\ar[dl]_{d}\ar[rd]^{c}\\\cdot&&\cdot}}
\end{equation}

This yields a reflection between the category of directed kites  and the category of spans

\[\xymatrix{\Kite \ar@<.5ex>[r]  &\Span\ar@<.5ex>[l]}\]

A directed kite goes to its direction span, a span goes to the directed kite displayed above. Moreover, the span $(D,d,c)$ is a pregroupoid if and only if its associated directed kite is multiplicative.

\item Consider a morphism between two directed kites as illustrated:
\begin{equation}\label{diag: kite6}
\vcenter{\xymatrix@!0@=4em{
A \ar@<.5ex>[r]^-{f} \ar[rd]_-{\alpha} & B
\ar@<.5ex>[l]^-{r}
\ar@<-.5ex>[r]_-{s}
\ar[d]^-{\beta} & C \ar@<-.5ex>[l]_-{g} \ar[ld]^-{\gamma}
&
A' \ar@<.5ex>[r]^-{f'} \ar[rd]_-{\alpha'} & B'
\ar@<.5ex>[l]^-{r'}
\ar@<-.5ex>[r]_-{s'}
\ar[d]^-{\beta'} & C' \ar@<-.5ex>[l]_-{g'} \ar[ld]^-{\gamma'}
\\
& D\ar[rrr]^{h_D}\ar[dl]_{d}\ar[rd]^{c}
&&& D'\ar[dl]_{d'}\ar[rd]^{c'}
\\
D_0&&D_1
& D'_0&&D'_1
}}
\end{equation}
The morphism $h_D$ is accompanied by additional morphisms, let us say
$h_0\colon{ D_0 \to D'_0}$, 
$h_1\colon D_1 \to D'_1$, 
$h_A\colon A \to A'$,  
$h_B\colon B \to B'$, 
$h_C\colon C \to C'$ 
which together satisfy the necessary compatibility conditions, making the full diagram commute. These data define what we refer to as a morphism of directed kites.

This morphism induces the following diagram:
\begin{equation}\label{diag: kite7}
\vcenter{\xymatrix@!0@=4em{
A \ar@<.5ex>[r]^-{f} \ar[rd]_-{h_D\alpha} & B
\ar@<.5ex>[l]^-{r}
\ar@<-.5ex>[r]_-{s}
\ar[d]^-{h_D\beta} & C \ar@<-.5ex>[l]_-{g} \ar[ld]^-{h_D\gamma}
\\
& D'\ar[dl]_{d'}\ar[rd]^{c'}
\\
D'_0&&D'_1
}}
\end{equation}
which is itself a directed kite.

Moreover, suppose this directed kite admits a unique multiplication, and that both of the original directed kites are multiplicative, with respective multiplication morphisms
\[
m\colon A \times_B C \to D
\quad \text{and} \quad
m'\colon A' \times_{B'} C' \to D'.
\]
Then the following compatibility condition holds
\[
h_D \circ m = m' \circ (h_A \times_{h_B} h_C),
\]
where $h_A \times_{h_B} h_C$ denotes the induced morphism from the pullback $A \times_B C$ to $A' \times_{B'} C'$.

\end{enumerate}
\end{proposition}

Given a class of spans \(\M\) in \(\C\), we may restrict our attention in all the previous definitions to those spans \((D,d,c)\) that belong to \(\M\). This includes the case of reflexive graphs: since their domain and codomain morphisms are split epimorphisms, \(d\) and \(c\) naturally form a span.

In each situation there is a canonical forgetful functor, as illustrated below.

\[
\xymatrix{
\Grpd(\C,\M)\ar[d]\ar[r]^{F_4} & \RG(\C,\M)\ar@{=}[d]\\
\Cat(\C,\M)\ar[d]\ar[r]^{F_3} & \RG(\C,\M)\ar@{=}[d]\\
\MG(\C,\M)\ar[r]^{F_2} & \RG(\C,\M)\ar@<-.5ex>[d]\\
\PreGrpd(\C,\M)\ar[u]\ar[d]\ar[r]^{F_1} & \Span(\C,\M)\ar@<-.5ex>[d]\ar@<-.5ex>[u] & \M\ar[l]_-{\cong}\\
\MKite(\C,\M)\ar[r]^{F_0} & \Kite(\C,\M)\ar@<-.5ex>[u]_{}}
\]

 For instance, the subcategory \(\RG(\C, \M)\) of reflexive graphs whose span part is in the class $\M$  determines a corresponding subcategory \(\MG(\C, \M)\) of unital multiplicative graphs, fitting into a pullback square:
\[
\xymatrix{
\MG(\C, \M) \ar[r]^{F^{\M}} \ar[d] & \RG(\C, \M) \ar[d] \\
\MG(\C) \ar[r]^F & \RG(\C)
}
\]
Here, \(F^{\M}\) denotes the restriction of \(F\) to those unital multiplicative graphs whose underlying span \((d, c)\) lies in \(\M\).

 Similarly, if the span part of a kite is required to be in $\M$ then it is an object in the  category $\Kite(\C,\M)$, where the morphisms are the natural transformations between such diagrams, as illustrated before. Further details may be found in \cite{MF17,MF38,MF2025}.


We have thus defined the relevant categories and the canonical functors between them that we will consider.

\section{Proof of the Main Theorem}\label{sec: proof of main thm}

In this section, we prove the equivalence among all the items appearing in the main theorem stated in the introduction (Theorem \ref{thm Main}, as well as its three corollaries). We include an additional item concerning unital multiplicative graphs to facilitate the proof (see also Definition~\ref{def: basics}, item~(10)). Later on, we will present an extended version of the theorem incorporating further related conditions (see Theorem \ref{thm Main extended}).

We are now in a position to restate and prove our main theorem.

\begin{theorem}\label{thm Main Lemma}
Let \( \C \) be a category with pullbacks of split epimorphisms along split epimorphisms, and let \( \M \) be a class of spans in \( \C \) that is closed under the kernel pair construction and contains the span part of every local product. Then the following conditions are equivalent:
\begin{enumerate}
\item The Lawvere Condition holds in $\C$ with respect to the class $\M$;
\item The forgetful functor $\Grpd(\C,\M)\to\RG(\C,\M)$ has a section;
\item The forgetful functor $\Cat(\C,\M)\to\RG(\C,\M)$ has a section;
\item The forgetful functor $\MG(\C,\M)\to\RG(\C,\M)$ has a section;
\item Every jointly monic span $(E,p_1,p_2)\in \M$ is $\M$-compatible;
\item Every dikite diagram in $\C$ whose direction span is in $\M$ is admissible;



\item The forgetful functor $\KockGrpd(\C,\M)\to\Span(\C,\M)$ is an isomorphism.



\end{enumerate}

\end{theorem}

\begin{proof}
We will show that:\\ $(1)\Rightarrow(2)\Rightarrow(3)\Rightarrow(4)\Rightarrow(5)\Rightarrow(6)\Rightarrow(7)\Rightarrow(1)$.

Clearly, (1) implies (2). Indeed, when we say that the Lawvere Condition holds in $\C$ with respect to the class $\M$, we are precisely saying that the forgetful functor $\Grpd(\C,\M)\to\Span(\C,\M)$ is an isomorphism, hence in particular it has a section.

Similarly, it is clear that (2) implies (3) which in turn implies (4).

 Let us prove (4) implies (5). Given a diagram in $\C$ of the form  $(\ref{diag: kite main thm (def) simple})$ as in Definition \ref{def: simple}, consider the diagram
\begin{equation}\label{diag: mu is natural}
\vcenter{\xymatrix{
E(p_1,p_2)\times_E E(p_1,p_2) \ar[dd]_{m^3\times_m m^3}\ar@<0ex>[r]^-{\mu_E} & E(p_1,p_2) \ar@<0ex>[r]^-{\midop_E} \ar@<0ex>[dd]^-{m^3} & E \ar@<0ex>@{-->}[dd]^-{m} \ar@/_3pc/[ll]_{\delta} \ar[rr]^{p_1}  \ar[ddll]_{\theta} \ar[rd]^{p_2} && A \ar[dd]^{c\alpha}
\\& & & C \ar[dd]^(.3){d\gamma}
\\
D(c,d)\times_D D(c,d) \ar[r]^-{\mu_D} & D(c,d) \ar[r]^-{\midop_D} & D \ar[rr]^(.3){c} \ar[rd]_{d} && D_1
\\ & & & D_0
}}
\end{equation}
where $\midop_D$ is obtained as in the kernel pair construction and $\mu_D$ is the multiplication on the reflexive graph $(D(c,d),\dom,\cod,\Delta)$ which is given by the section to the forgetful functor from unital multiplicative graphs whose span part is in $\M$ to the category of reflexive graphs whose span part is in $\M$
\[
\xymatrix@C=4em{
D(c,d)\times_D D(c,d)
  \ar[r]^-{\mu_D}
& D(c,d)
  \ar@<1.2ex>[r]^-{\dom_D}
  \ar@<-1.5ex>[r]|-{\cod_D}
  \ar@<-2.2ex>@/_/[r]_-{\midop_D}
& D
  \ar[l]|-{\Delta_D}
}
\]
similarly we obtain (since $(E,p_1,p_2)$ is in $\M$)
\[
\xymatrix@C=4em{
E(p_1,p_2)\times_E E(p_1,p_2)
  \ar[r]^-{\mu_E}
& E(p_1,p_2)
  \ar@<1.2ex>[r]^-{\dom_E}
  \ar@<-1.5ex>[r]|-{\cod_E}
  \ar@<-2.2ex>@/_/[r]_-{\midop_E}
& E
  \ar[l]|-{\Delta_E}
}
\]
The morphisms $\delta$ and $\theta$ are well defined and are of the form
\begin{align*}
\theta=\langle\langle\alpha p_1,\alpha p_1, \alpha p_1 e_2 p_2\rangle,\langle\gamma p_2 e_1 p_1, \gamma p_2, \gamma p_2\rangle\rangle\\
\delta=\langle\langle e_1p_1,e_1p_1,e_1p_1e_2p_2\rangle,\langle e_2p_2e_1p_1,e_2p_2,e_2p_2\rangle
\rangle.
\end{align*}
For the existence we take
\begin{align}
m=\midop_D\circ \mu_D\circ \theta
\end{align}
whereas for the uniqueness we first observe that:
\begin{align*}
\text{the pair $(p_1,p_2)$ is jointly monic}\\
p_1\circ \midop_E\circ \mu_E \circ \delta=p_1\\
p_2\circ \midop_E\circ \mu_E \circ \delta=p_2\\
\end{align*}
and hence $1_E=\midop_E\circ \mu_E\circ \delta$, and moreover, $(m^3\times_m m^3)\circ \delta=\theta$, which, together with the hypoteses that $(E, p_1,p_2)\in \M$, gives, by naturality of $\midop$ and $\mu$ 
\begin{align*}
\midop_D\circ\mu_D\circ(m^3\times_m m^3)=m\circ\midop_E\mu_E
\end{align*} 
and as a consequence
\begin{align*}
m&=m\circ 1_E= m\circ \midop_E\circ \mu_E\circ \delta\\
&=\midop_D\circ\mu_D\circ(m^3\times_m m^3)\circ \delta\\
&=\midop_D\circ \mu_D\circ \theta.
\end{align*}
This means that $m$ is uniquely determined as desired.

In order to prove  $(5)\Rightarrow  (6)$ it suffices to observe from Definition \ref{def: simple} that the diagram $(\ref{diag: dikite (d,c) in def simple})$ induces a diagram of the form $(\ref{diag: kite main thm (def) simple})$ in which $E=A\times_B C$, $p_1=\pi_1$ and $p_2=\pi_2$ are the two canonical projections from the pullback of $f$ and $g$, whereas $e_1=\langle 1_A,sf\rangle$ and $e_2=\langle rg,1_C\rangle$ are the two canonical morphisms induced by the sections $r$ and $s$, respectively, of $f$ and $g$. It is now a simple observation to check that all conditions are satisfied and hence the desired morphism $m\colon A\times_B C\to D$ is obtained.  

That (6) implies (7) is a consequence of Proposition \ref{prop: particular cases}. Indeed, if every directed kite whose direction span is in the class $\M$ is admissible then, in particular, every span in $\M$ has a unique canonical pregroupoid structure which is associative.

Finally, (7) implies (1) is a well-known fact. Take any reflexive graph $(C_1,C_0,d,e,c)$ with $(C_1,d,c)\in \M$, consider its pregroupoid structure $$p_{C_1}\colon{C_1(d,c)\to C_1}$$ and set $m=p_{C_1}\langle 1_{C_1}, ed, 1_{C_1}\rangle$ and $i=p_{C_1}\langle 1_{ed,C_1}, ec\rangle$ so that $$(C_1,C_0,d,e,c,m,i)$$ becomes an internal groupoid. Uniqueness follows from the naturality of the canonical morphism $p_{C_1}\colon{C_1(d,c)\to C_1}$.
\end{proof}


\section{Proof of the Corollaries}

Let us now prove the three corollaries stated in the introduction.

\subsection{Weakly Mal'tsev}

We will show that for a category \(\C\) with pullbacks of split epimorphisms along split epimorphisms and equalizers, the category \(\C\) is weakly Mal'tsev if and only if the relative Lawvere condition holds in \(\C\) with respect to the class of jointly strongly monic spans.

One direction is straightforward. Indeed, if \(\C\) is weakly Mal'tsev, then every kite diagram whose direction span \((D, d, c)\) lies in \(\M\) is admissible. The required morphism is obtained as the unique diagonal filler arising from the fact that a jointly epic cospan is orthogonal to every jointly strongly monic span, as illustrated by the square
\[
\xymatrix{
A + C \ar[r]^{[e_1, e_2]} \ar[d]_{\langle \alpha, \gamma \rangle} 
  & A \times_B C \ar[d]^{\langle d\gamma\pi_2,\, c\alpha\pi_1 \rangle} \ar@{-->}[ld]_m \\
D \ar[r]_{\langle d, c \rangle} & D_0 \times D_1
}
\]
in the case where the product \(D_0 \times D_1\) and the coproduct \(A + C\) exist. Clearly, this assumption is not essential, since we are working with jointly epic cospans and jointly strongly monic spans.

Conversely, given any local product diagram of the form
\begin{equation}
\vcenter{\xymatrix{
A \ar@<-.5ex>[r]_-{e_1} & E \ar@<-.5ex>[l]_-{p_1} \ar@<.5ex>[r]^-{p_2} & C \ar@<.5ex>[l]^-{e_2}
}}
\end{equation}
obtained as the pullback of \(f\) and \(g\) in
\begin{equation}
\vcenter{\xymatrix@!0@=4em{
A \ar@<.5ex>[r]^-{f} & B \ar@<.5ex>[l]^-{r} \ar@<-.5ex>[r]_-{s} & C \ar@<-.5ex>[l]_-{g}
}}
\end{equation}
with \( f r = 1_B = g s \), our task is to prove that the pair \((e_1, e_2)\) is jointly epic.  

To this end, consider the equalizer of any two parallel morphisms \[u, v \colon E \to Z\] satisfying \(u e_1 = v e_1\) and \(u e_2 = v e_2\). Let \(\overline{e_1}\) and \(\overline{e_2}\) be the morphisms induced by \(e_1\) and \(e_2\), respectively, as displayed:
\[
\xymatrix{
C \ar@{-->}[d]_{\overline{e_2}} \ar[rrd]^{e_2} \\
K \ar[rr]^{k = \mathrm{eq}(u,v)} && E \ar@<.5ex>[r]^{u} \ar@<-.5ex>[r]_{v} & Z \\
A \ar@{-->}[u]^{\overline{e_1}} \ar[rru]_{e_1}
}
\]

We can now form the kite diagram
\begin{equation}
\vcenter{\xymatrix@!0@=4em{
A \ar@<.5ex>[r]^-{f} \ar[rd]_-{\overline{e_1}} & B
\ar@<.5ex>[l]^-{r}
\ar@<-.5ex>[r]_-{s}
\ar[d]^-{} & C \ar@<-.5ex>[l]_-{g} \ar[ld]^-{\overline{e_2}} \\
& K \ar[dl]_{p_2 k} \ar[rd]^{p_1 k} \\
C && A
}}
\end{equation}
and observe that the span \((K, p_2 k, p_1 k) \in \M\).  
By admissibility, we obtain a morphism \(E \to K\), which turns the monomorphism \(k \colon K \to E\) into an isomorphism. Consequently, \(u = v\), and hence the pair \((e_1, e_2)\) is jointly epimorphic.

\subsection{Mal'tsev}
We omit the proof of the classical case of a Mal'tsev category, as it is well documented in the literature. Even though we do not assume the existence of all finite limits, the arguments proceed in close analogy with the standard proofs; see, for example, \cite{Bournetall} and the references therein.

\subsection{Naturally Mal'tsev}

Suppose that a category \(\C\) is naturally Mal'tsev. Then any span \((D, d, c)\) in \(\C\) such that \(d\) and \(c\) admit kernel pairs can be equipped with a natural pregroupoid structure
\[
p = p_D\langle \dom, \midop, \cod \rangle \colon D(d, c) \to D,
\]
where \(p_D \colon D^3 \to D\) is the natural Mal'tsev operation on the object \(D\), and \(\dom, \midop, \cod \colon D(d, c) \to D\) are the projections obtained by the kernel pair construction, as in Definition~\ref{def: basics}, item~(3).

Conversely, every object \(X\) in \(\C\) inherits a Mal'tsev operation
\[
p_X \colon X^3 \to X
\]
by admissibility of the dikite diagram
\begin{equation}
\vcenter{\xymatrix@!0@=4em{
X \times X \ar@<.5ex>[r]^-{\pi_2} \ar[rd]_-{\pi_1} & X
\ar@<.5ex>[l]^-{\Delta}
\ar@<-.5ex>[r]_-{\Delta}
\ar@{=}[d]^-{} & X \times X \ar@<-.5ex>[l]_-{\pi_1} \ar[ld]^-{\pi_2} \\
& X \ar[dl]_{} \ar[rd]^{} \\
1 && 1
}}
\end{equation}
which is natural by admissibility of every diagram of the form
\begin{equation}
\vcenter{\xymatrix@!0@=4em{
X \times X \ar@<.5ex>[r]^-{\pi_2} \ar[rd]_-{f\pi_1} & X
\ar@<.5ex>[l]^-{\Delta}
\ar@<-.5ex>[r]_-{\Delta}
\ar[d]^-{f} & X \times X \ar@<-.5ex>[l]_-{\pi_1} \ar[ld]^-{f\pi_2} \\
& Y \ar[dl]_{} \ar[rd]^{} \\
1 && 1
}}
\end{equation}
for every morphism \(f \colon X \to Y\).
This shows that the forgetful functor from pregroupoids to spans admits a section. The result then follows from the extended version of the main theorem, presented in the next section.

\section{Extending the scope of the main theorem}\label{sec: extending main thm}

In this section, we include a few additional remarks related to the main theorem presented in the introduction, the most notable of which is the connection with pseudogroupoids and, consequently, with categorical commutator theory~\cite{Gran2002,Janelidze2016,JK}. For the box construction and related considerations, see Definition~\ref{def: basics}, item~(2). Recall that $\MG(\C,\M)$ denotes the category of unital multiplicative graphs in $\C$ whose span part is in the class $\M$. 


\begin{theorem}\label{thm Main extended}
Let \( \C \) be a category admitting pullbacks of split epimorphisms along split epimorphisms, and let \( \M \) denote a class of spans in \( \C \) for which the box construction is defined. Assume further that \( \M \) is closed under the kernel pair construction and contains the span part of every local product.

 Then the following conditions are equivalent:
\begin{enumerate}
\item[(A)] The Lawvere Condition holds in $\C$ with respect to the class $\M$;
\item[(B)] The forgetful functor $\Grpd(\C,\M)\to\RG(\C,\M)$ has a section;
\item[(C)] The forgetful functor $\Cat(\C,\M)\to\RG(\C,\M)$ is an isomorphism;
\item[(D)] The forgetful functor $\Cat(\C,\M)\to\RG(\C,\M)$ has a section;
\item[(E)] The forgetful functor $\MG(\C,\M)\to\RG(\C,\M)$ is an isomorphism;
\item[(F)] The forgetful functor $\MG(\C,\M)\to\RG(\C,\M)$ has a section;

\item[(G)] For every diagram in \( \C \) of the form
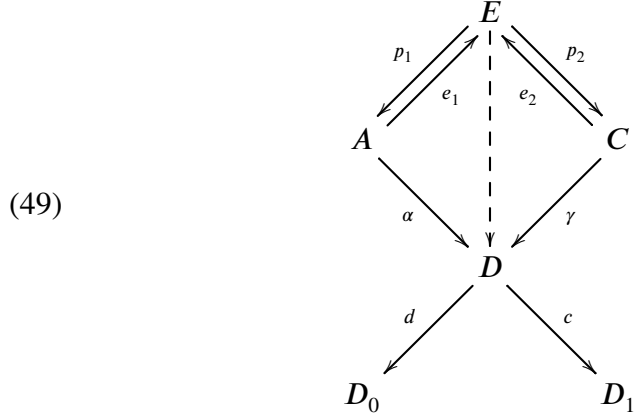
\begin{equation}\label{diag: kite main thm}
\vcenter{\xymatrix@!0@=4em{
& E \ar@{-->}[dd]
\ar@<-.5ex>[dl]_-{p_1}
\ar@<.5ex>[dr]^-{p_2}
\\
A \ar@<-.5ex>[ru]_-{e_1} \ar[rd]_-{\alpha} 
&  & 
C \ar@<.5ex>[lu]^-{e_2} \ar[ld]^-{\gamma} \\
& D \ar[dl]_{d} \ar[rd]^{c} \\
D_0 && D_1
}}
\end{equation}
if:
\begin{enumerate}
    \item the pair \( (p_1, p_2) \) is jointly monic;
    \item the spans \( (D, d, c) \) and \( (E, p_1, p_2) \) belong to \( \M \);
    \item the following identities hold:
    \begin{align*}
    & p_1 e_1 = 1_A, \quad p_2 e_2 = 1_C, \\
    & (e_1 p_1)(e_2 p_2) = (e_2 p_2)(e_1 p_1), \\
    & \alpha p_1 (e_2 p_2) = \gamma p_2 (e_1 p_1), \\
    & d \alpha p_1 = d \alpha p_1 (e_2 p_2), \quad \\ & c \gamma p_2 = c \gamma p_2 ( e_1 p_1),
    \end{align*}
\end{enumerate}
 then there exists a unique morphism \( m \colon E \to D \) such that:
\begin{align}
m e_1 = \alpha, \quad
m e_2 = \gamma, \quad
dm =d\gamma p_2, \quad cm =c\alpha p_1.    \end{align}

\item[(H)] For every diagram in \( \C \) of the form
\begin{equation}\label{diag: dikite (d,c)}
\vcenter{\xymatrix@!0@=4em{
A \ar@<.5ex>[r]^-{f} \ar[rd]_-{\alpha} 
& B \ar[d]^-{\beta}
\ar@<.5ex>[l]^-{r}
\ar@<-.5ex>[r]_-{s} & 
C \ar@<-.5ex>[l]_-{g} \ar[ld]^-{\gamma} \\
& D \ar[dl]_{d} \ar[rd]^{c} \\
D_0 && D_1
}}
\end{equation}
if:
\begin{align}
& fr=1_B=gs\\
& \alpha r=\beta=\gamma s\\
& d\alpha=d\beta f\\
& c\beta g=c\gamma
\end{align}
and \[ (D,d,c)\in \M\]
 then there exists a unique morphism \[ m \colon A\times_B C \to D \] such that:
\begin{align*}
& m \langle 1_A, sf \rangle = \alpha, \quad \\
& m \langle rg,1_C \rangle = \gamma,\quad \\
& dm =d\gamma \pi_2 ,\quad \\ 
& cm =c\alpha \pi_1.    \end{align*}


\item[(I)] The forgetful functor $\mathbf{PseudoGrpd}(\C,\M)\to\Span(\C,\M)$ has a section, and, moreover, if $(D,d,c)$ is a span in $\M$ and 
\begin{equation}\label{split square main intro}
\xymatrix@!0@=4em{
E \ar@<.5ex>[r]^-{p_2} \ar@<-.5ex>[d]_-{p_1} & C \ar@<.5ex>[l]^-{e_2} \ar@<-.5ex>[d]_-{g} \\
A \ar@<.5ex>[r]^-{f} \ar@<-.5ex>[u]_-{e_1} & B \ar@<.5ex>[l]^-{r} \ar@<-.5ex>[u]_-{s}.
}
\end{equation}
is a commutative split square in $\C$ then for every morphism $$m\colon{E\to D}$$ with $dm=dme_2p_2$, $cm=cme_1p_1$, there exists a unique morphism $$\overline{m}\colon{A\times_B C\to D}$$ such that 
\begin{align*}
\overline{m}\langle 1_A,sf\rangle=me_1\\
\overline{m}\langle rg,1_C\rangle=me_2\\
d\overline{m}=dme_2\pi_2\\
c\overline{m}=cme_1\pi_1.
\end{align*}

\item[(J)] Every span $(D,d,c)\in \M$ is equipped with a morphism
\begin{align*}
p_D\colon{D(d,c)\to D}
\end{align*}
 such that given any three parallel morphisms $x,y,z\colon{Z\to D}$ with $dx=dy$ and $cy=cz$ the following conditions hold \begin{align*}
dp_D\langle x,y,z\rangle=dz,\quad cp_D\langle x,y,z\rangle=cx,\\
p_D\langle x,y,y\rangle=x ,\quad p_D\langle y,y,z\rangle=z,
\end{align*}
and moreover, given any commutative diagram of the form 
\begin{align*}
\xymatrix{D_0 \ar[d]_{f_0} & D \ar[d]^{f} \ar[l]_{d} \ar[r]^{c} & D_1\ar[d]^{f_1}\\D'_0  & D' \ar[l]_{d'} \ar[r]^{c'} & D'_1}
\end{align*} 
with $(D',d',c')\in \M$, the square
\begin{align*}
\xymatrix{D(d,c)\ar[r]^-{p_D}\ar[d]_{f\times_{f_0} f\times_{f_1} f} & D\ar[d]^{f}\\
D'(d',c') \ar[r]^-{p'_{D'}} & D'}
\end{align*}
is a commutative square. 

\item[(K)] The forgetful functor $\PreGrpd(\C,\M)\to\Span(\C,\M)$ is an isomorphism;
\item[(L)] The forgetful functor $\KockGrpd(\C,\M)\to\Span(\C,\M)$ has a section;
\item[(M)] The forgetful functor $\KockGrpd(\C,\M)\to\Span(\C,\M)$ is an isomorphism;
\item[(N)] The forgetful functor $\MKite(\C,\M)\to\Kite(\C,\M)$ has a section;
\item[(O)] The forgetful functor $\MKite(\C,\M)\to\Kite(\C,\M)$ is an isomorphism;

\end{enumerate}

\end{theorem}

\begin{proof}
The proof is asserted by establishing the following implications:
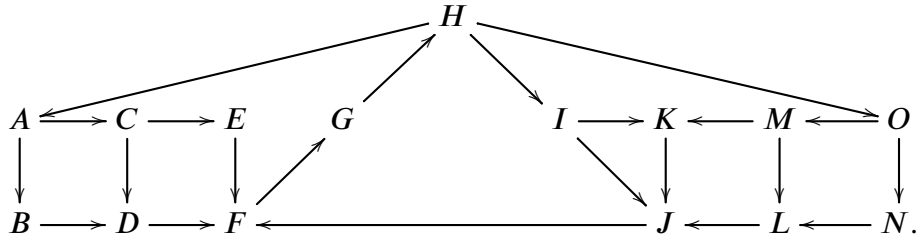
\begin{align*}
\xymatrix{&&&& H\ar[lllld]\ar[rrrrd]\ar[rd]\\
A\ar[r]\ar[d]&C\ar[r]\ar[d]& E\ar[d] & G\ar[ur] && I\ar[r]\ar[rd]&K\ar[d]&M\ar[l]\ar[d]&O\ar[l]\ar[d]\\
B\ar[r]&D\ar[r]&F\ar[ur]&&&& J\ar[llll]&L\ar[l]&N.\ar[l]}
\end{align*}
The implications in
\begin{align*}
\xymatrix{
A\ar[r]\ar[d]&C\ar[r]\ar[d]& E\ar[d] \\
B\ar[r]&D\ar[r]&F}
\end{align*}\text{and}
\begin{align*}
\xymatrix{
K\ar[d]&M\ar[l]\ar[d]&O\ar[l]\ar[d]\\
 J&L\ar[l]&N\ar[l]}
\end{align*}
\end{proof}
are clear. Let us prove:
\begin{enumerate}
\item[(H,I)] We have already seen that condition (H) implies that the forgetful functor $\PreGrpd(\C,\M)\to \Span(\C,\M)$ has a section, hence, in particular, $\mathbf{PseudoGrpd}(\C,\M)\to \Span(\C,\M)$ has a section as well. Furthermore, given any span $(D,d,c)\in \M$ and a commutative square as illustrated in condition (I), together with $m\colon{E\to D}$ satisfying $dm=dme_2p_2$ and $cm=cme_1p_1$ we can build the dikite diagram
\begin{equation}\label{diag: dikite (d,c) proof of ext main thm}
\vcenter{\xymatrix@!0@=4em{
A \ar@<.5ex>[r]^-{f} \ar[rd]_-{me_1} 
& B \ar[d]^-{\beta}
\ar@<.5ex>[l]^-{r}
\ar@<-.5ex>[r]_-{s} & 
C \ar@<-.5ex>[l]_-{g} \ar[ld]^-{me_2} \\
& D \ar[dl]_{d} \ar[rd]^{c} \\
D_0 && D_1
}}
\end{equation}
in which:
\begin{align}
& fr=1_B=gs\\
& me_1 r=\beta=me_2 s\\
& dme_1=d\beta f\\
& c\beta g=cme_2
\end{align}
and  then, by admissibility, there exists a unique morphism \[ m \colon A\times_B C \to D \] such that:
\begin{align*}
& m \langle 1_A, sf \rangle = me_1, \quad \\
& m \langle rg,1_C \rangle = me_2,\quad \\
& dm =dme_2 \pi_2 ,\quad \\ 
& cm =cme_1 \pi_1.    \end{align*}

 \item[(I,K,J)] 
Given $(D,d,c)\in \M$, consider $m\colon D(\square)\to D$, its associated pseudogroupoid. 

Next, consider the commutative split square
\begin{equation}\label{split square main ext  them proof}
\xymatrix@!0@=4em{
D(\square) \ar@<.5ex>[r]^-{p_2} \ar@<-.5ex>[d]_-{p_1} & D(c) \ar@<.5ex>[l]^-{e_2} \ar@<-.5ex>[d]_-{c_1} \\
D(d) \ar@<.5ex>[r]^-{d_2} \ar@<-.5ex>[u]_-{e_1} & D \ar@<.5ex>[l]^-{\Delta} \ar@<-.5ex>[u]_-{\Delta}
}
\end{equation}
From this square, we obtain a morphism
\[
\overline{m}\colon D(d,c)\to D,
\]
which defines a unique pregroupoid structure on the span $(D,d,c)$. In other words, every span $(D,d,c)\in \M$ is equipped with a unique pregroupoid structure, denoted $p_D = \overline{m}$.

To prove the naturality of each $p_D$, observe that the following diagram is a split coequalizer:
\begin{align*}
\xymatrix{
D(\underline{\square}) \ar@<.5ex>[r]^{r_1} \ar@<-.5ex>[r]_{r_2} 
& D(\square) \ar[r]^{\langle q_1,q_2,q_3\rangle} \ar@/_2pc/[l]_{t} 
& D(d,c) \ar@/_2pc/[l]_{s}
}
\end{align*}
Here, $r_1,r_2\colon D(\underline{\square})\to D(\square)$ form the kernel pair of $\langle q_1,q_2,q_3\rangle$, which is assumed to exist. 

Intuitively, the object $D(\underline{\square})$ consists of tuples $\langle x,y,z,w,w'\rangle$, as illustrated by
\[
\xymatrix{
\cdot \ar[r]^{y} \ar[d]_{x} & \cdot \\
\cdot & \cdot \ar@<.5ex>@{-->}[l]^{w'} \ar@<-.5ex>@{-->}[l]_{w} \ar[u]_{z}
}
\]

Finally, the maps $t$ and $s$ are defined on generalized elements by
\begin{align*}
s\langle x,y,z\rangle &= \langle x,y,z,\overline{m}\langle x,y,z\rangle\rangle,\\
t\langle x,y,z,w\rangle &= \langle x,y,z,w,\overline{m}\langle x,y,z\rangle\rangle.
\end{align*}

Now, since $m$ satisfies $mr_1=mr_2$ (the pseudogroupoid structure $m$ is independent of its fourth variable, see Definition \ref{def: basics}), 
there exists a unique $\widetilde{m}\colon D(d,c)\to D$ such that $\widetilde{m}\langle q_1,q_2,q_3\rangle = m$. 

By the uniqueness property that defines $\overline{m}$, we have $\widetilde{m} = \overline{m}$, and therefore 
$m = \overline{m}\langle q_1,q_2,q_3\rangle$ as well.

The fact that $\langle q_1,q_2,q_3\rangle\colon D(\square)\to D(d,c)$ is an epimorphism, together with the commutativity of the outer square in the diagram
\begin{align*}
\xymatrix{D(\square) \ar[d]_{f^4}\ar@/^2pc/[rr]^{m}\ar[r]^{\langle q_1,q_2,q_3\rangle} & D(d,c)\ar[d]_{f^3} \ar[r]^{\overline{m}} & D \ar[d]_{f}\\
D'(\square) \ar@/_2pc/[rr]_{m'}\ar[r]^{\langle q'_1,q'_2,q'_3\rangle} & D'(d',c') \ar[r]^{\overline{m'}} & D' }
\end{align*}
gives the desired naturality of $p_D = \overline{m}$.

\item[(F,G,H)] Has been established in Section \ref{sec: proof of main thm}, as well as (J,F).


\item[(H,A)] Follows from Proposition \ref{prop: particular cases}, as well as (H,O).

\end{enumerate}

\section{A few further applications}\label{section: applications}

In this last section we outline a few simple applications that are consequence of the extended main theorem.

Weakly Mal'tsev categories generalize Mal'tsev categories, capturing situations in which not every internal category forms a groupoid. This occurs, for example, in the categories of preordered groups, distributive lattices, or commutative monoids with cancellation \cite{MF37,MF26,MF9,MF9a,MF24,MF47,MF2025}.

\begin{proposition}\label{prop: WMC}
Let $\C$ be a category with pullbacks of split epimorphisms along split epimorphisms and with equalizers. Let $\M$ be a class of spans in $\C$ that is closed under the kernel pair construction, contains the span part of every local product, and is stable under restriction along regular monomorphisms, in the sense that for every $(D,d,c)\in \M$ and every regular monomorphism $k\colon E\to D$, the restricted span $(E,dk,ck)$ is also in $\M$.

The Lawvere condition holds in $\C$ with respect to $\M$, if and only if  $\C$ is a weakly Mal'tsev category.
\end{proposition}

As a consequence, weakly Mal'tsev categories can be characterized by taking \(\M\) to be the smallest class of spans containing all exact spans and closed under restriction along regular monomorphisms. Here, exact spans are those that admit a pushout and are the pullback of that pushout.

When all finite limits exist, this smallest class coincides precisely with the class of strong relations. Hence, weakly Mal'tsev categories can be seen as the broadest context in which the characteristic properties of Mal'tsev-like categories are preserved.

On the other hand, it remains open whether other types of classes may arise—situations in which not all spans are relations, and not all induced cospans into local products are jointly epic, but where the interaction between spans and cospans uniquely determines morphisms between them. More precisely, instead of requiring a jointly epic cospan \((e_1, e_2)\) or a jointly monic span \((d, c)\), one might have both cooperating in such a way that, for any \(u, v \colon E \to D\), the conditions
\[
ue_1 = ue_2, \quad ve_1 = ve_2, \quad du = dv, \quad \text{and} \quad cu = cv
\]
together imply \(u = v\).


Another consequence of the relative Lawvere condition is a simplification in the characterization of local products.

\begin{proposition}\label{prop: WMC local prods}
Let $\C$ be a category with pullbacks of split epimorphisms along split epimorphisms, and let $\M$ be a class of spans in $\C$ that is closed under the kernel pair construction and contains the span part of every local product.

If the Lawvere condition holds in $\C$ with respect to $\M$, then a diagram in $\C$ of the form
\begin{equation}\label{local product 0.2}
\vcenter{\xymatrix{
A \ar@<-.5ex>[r]_-{e_1} & E \ar@<-.5ex>[l]_-{p_1} \ar@<.5ex>[r]^-{p_2} & C \ar@<.5ex>[l]^-{e_2}
}}
\end{equation}
is a local product if and only if the following conditions hold:
\begin{enumerate}
\item \( p_1 e_1 = 1_A \) and \( p_2 e_2 = 1_C \);
\item \( e_1 p_1 e_2 p_2 = e_2 p_2 e_1 p_1 \);
\item the pair \( (p_1, p_2) \) is jointly monic;
\item the span $(E,p_1,p_2)$ belongs to $\M$;
\item the pullback of \( e_1 \) and \( e_2 \) exists.
\end{enumerate}
\end{proposition}

\begin{proof}
If the diagram is a local product, then by Definition \ref{def: basics} item (8), all the conditions are satisfied. In particular, the span $(E,p_1,p_2)$ belongs to $\M$ by hypothesis. 

Conversely, using again Definition \ref{def: basics} item (8), as well as Proposition 1 in \cite{MF2025}, it suffices to verify that for every pair of morphisms \( u\colon Z \to A \) and \( v\colon Z \to C \), if
\begin{align*}
p_1 e_2 p_2 e_1 u &= p_1 e_2 v, \\
p_2 e_1 u &= p_2 e_1 p_1 e_2 v,
\end{align*}
then there exists a unique morphism \( w\colon Z \to E \) such that \( p_1 w = u \) and \( p_2 w = v \). This follows from the relative Lawvere condition.
\end{proof}

Another application is to identify the conditions under which a canonical pregroupoid structure on a span in \(\M\) is autonomous (see Definition~\ref{def: basics}, item~(15)).

\begin{proposition}\label{prop: autonomous}
Let \(\C\) be a category with pullbacks of split epimorphisms along split epimorphisms, and let \(\M\) be a class of spans in \(\C\) that is closed under the kernel pair construction and contains the span part of every local product.

Suppose the Lawvere condition holds in \(\C\) with respect to \(\M\). A span \((D,d,c) \in \M\) is \emph{autonomous}—that is, the canonical induced pregroupoid \((D,d,c,p_D)\) is autonomous—whenever the products \(D_0 \times D_0\) and \(D_1 \times D_1\) exist in \(\C\), and not only the span \((D(d,c), \dom, \cod) \in \M\) but also the span \((D(d,c), \overline{d}, \overline{c})\) belongs to \(\M\), where
\begin{align*}
\overline{d} &= \langle d \circ \dom,\, d \circ \cod \rangle \colon D(d,c) \to D_0 \times D_0,\\
\overline{c} &= \langle c \circ \dom,\, c \circ \cod \rangle \colon D(d,c) \to D_1 \times D_1.
\end{align*}
\end{proposition}

\begin{proof}
Assuming that $(D(d,c),\overline{d},\overline{c})\in \M$ we can form the kernel pair construction and consequently consider the diagram
\begin{equation}\label{diag: kite prop autonomous}
\vcenter{\xymatrix@!0@=4em{
& D(d,c)(\overline{d},\overline{c}) 
\ar@<-.5ex>[dl]_-{p_1}
\ar@<.5ex>[dr]^-{p_2}
\ar@{-->}[dd]
\\
D(d,c)(\overline{d}) \ar@<-.5ex>[ru]_-{e_1} \ar[rd]_-{\overline{d}_1} 
&  & 
D(d,c)(\overline{c}) \ar@<.5ex>[lu]^-{e_2} \ar[ld]^-{\overline{c}_2} \\
& D(d,c) \ar[dl]_{\overline{d}} \ar[rd]^{\overline{c}} \\
D_0\times D_0 && D_1\times C_1
}}
\end{equation}
and observe that both morphisms
\begin{align*}
p_{D(d,c)}=\langle p_{D(d,c)}^{1}, p_{D(d,c)}^{2}, p_{D(d,c)}^{3}\rangle\\
\langle p_D\langle x_1,y_1,z_1\rangle,p_D\langle x_2,y_2,z_2\rangle,p_D\langle x_3,y_3,z_3\rangle\rangle
\end{align*}
fit into the diagram. Hence they must be equal. The naturality of $p_D$ and $p_{D(d,c)}$ completes the proof, as displayed
\begin{equation}\label{diag: p_D is natural }
\vcenter{\xymatrix{
D(d,c)(\overline{d},\overline{c}) \ar[dd]_{p_D^{3}}\ar@<0ex>[rr]^-{p_{D(d,c)}}  && D(d,c) \ar@<0ex>@{->}[dd]^-{p_D}  \ar[rr]^{\overline{c}}   \ar[rd]^{\overline{d}} && D_1\times D_1 \ar[dd]^{\pi_1}
\\& & & D_0\times D_0 \ar[dd]^(.3){\pi_2}
\\
D(c,d) \ar[rr]^-{p_D} && D  \ar[rr]^(.3){c} \ar[rd]_{d} && D_1
\\ & & & D_0
}}
\end{equation}
\end{proof}


If we take \(\M\) to be the class of all spans in \(\C\), the extended version of the main theorem yields a characterization of naturally Mal'tsev categories, now including a few additional equivalent conditions.

\begin{theorem}\label{thm Main naturally}
Let \(\C\) be a category with pullbacks of split epimorphisms along split epimorphisms, kernel pairs, and finite products. Then the following conditions are equivalent:

\begin{enumerate}
\item The category $\C$ is naturally Mal'tsev.
\item The Lawvere Condition holds in $\C$.
\item The functor $\Grpd(\C)\to\RG(\C)$ has a section;
\item The functor $\Cat(\C)\to\RG(\C)$ is an isomorphism;
\item The functor $\Cat(\C)\to\RG(\C)$ has a section;
\item The functor $\UMG(\C)\to\RG(\C)$ is an isomorphism;
\item The functor $\UMG(\C)\to\RG(\C)$ has a section;
\item The functor $\PreGrpd(\C)\to\Span(\C)$ is an isomorphism;
\item The functor $\PreGrpd(\C)\to\Span(\C)$ has a section;
\item The functor $\KockGrpd(\C)\to\Span(\C)$ is an isomorphism;
\item The functor $\KockGrpd(\C)\to\Span(\C)$ has a section;
\item For every diagram in \( \C \) of the form
\begin{equation}\label{diag: dikite (d,c) naturally}
\vcenter{\xymatrix@!0@=4em{
A \ar@<.5ex>[r]^-{f} \ar[rd]_-{\alpha} 
& B \ar[d]^-{\beta}
\ar@<.5ex>[l]^-{r}
\ar@<-.5ex>[r]_-{s} & 
C \ar@<-.5ex>[l]_-{g} \ar[ld]^-{\gamma} \\
& D \ar[dl]_{d} \ar[rd]^{c} \\
D_0 && D_1
}}
\end{equation}
if:
\begin{align}
& fr=1_B=gs\\
& \alpha r=\beta=\gamma s\\
& d\alpha=d\beta f\\
& c\beta g=c\gamma
\end{align}
 then there exists a unique morphism \[ m \colon A\times_B C \to D \] such that:
\begin{align*}
& m \langle 1_A, sf \rangle = \alpha, \quad \\
& m \langle rg,1_C \rangle = \gamma,\quad \\
& dm =d\gamma \pi_2 ,\quad \\ 
& cm =c\alpha \pi_1.    \end{align*}

\item For every diagram in \( \C \) of the form
\begin{equation}\label{diag: kite main thm}
\vcenter{\xymatrix@!0@=4em{
& E \ar@{-->}[dd]
\ar@<-.5ex>[dl]_-{p_1}
\ar@<.5ex>[dr]^-{p_2}
\\
A \ar@<-.5ex>[ru]_-{e_1} \ar[rd]_-{\alpha} 
&  & 
C \ar@<.5ex>[lu]^-{e_2} \ar[ld]^-{\gamma} \\
& D \ar[dl]_{d} \ar[rd]^{c} \\
D_0 && D_1
}}
\end{equation}
if:
\begin{enumerate}
    \item the pair \( (p_1, p_2) \) is jointly monic;
    \item the following identities hold:
    \begin{align*}
    & p_1 e_1 = 1_A, \quad p_2 e_2 = 1_C, \\
    & (e_1 p_1)(e_2 p_2) = (e_2 p_2)(e_1 p_1), \\
    & \alpha p_1 (e_2 p_2) = \gamma p_2 (e_1 p_1), \\
    & d \alpha p_1 = d \alpha p_1 (e_2 p_2), \quad \\ & c \gamma p_2 = c \gamma p_2 ( e_1 p_1),
    \end{align*}
\end{enumerate}
 then there exists a unique morphism \( m \colon E \to D \) such that:
\begin{align}
m e_1 = \alpha, \quad
m e_2 = \gamma, \quad
dm =d\gamma p_2, \quad cm =c\alpha p_1.    \end{align}


\item[(I)] The forgetful functor $\mathbf{PseudoGrpd}(\C)\to\Span(\C)$ has a section, and, moreover, if $(D,d,c)$ is a span in $\C$ and 
\begin{equation}\label{split square main intro}
\xymatrix@!0@=4em{
E \ar@<.5ex>[r]^-{p_2} \ar@<-.5ex>[d]_-{p_1} & C \ar@<.5ex>[l]^-{e_2} \ar@<-.5ex>[d]_-{g} \\
A \ar@<.5ex>[r]^-{f} \ar@<-.5ex>[u]_-{e_1} & B \ar@<.5ex>[l]^-{r} \ar@<-.5ex>[u]_-{s}.
}
\end{equation}
is a commutative split square in $\C$ then for every morphism $$m\colon{E\to D}$$ with $dm=dme_2p_2$, $cm=cme_1p_1$, there exists a unique morphism $$\overline{m}\colon{A\times_B C\to D}$$ such that 
\begin{align*}
\overline{m}\langle 1_A,sf\rangle=me_1\\
\overline{m}\langle rg,1_C\rangle=me_2\\
d\overline{m}=dme_2\pi_2\\
c\overline{m}=cme_1\pi_1.
\end{align*}

\item The induced comparison morphism $\langle p_1,p_2\rangle\colon{E\to A\times_{B} C}$ in every split square
\begin{equation}\label{split square nat}
\xymatrix@!0@=4em{
E \ar@<.5ex>[r]^-{p_2} \ar@<-.5ex>[d]_-{p_1} & C \ar@<.5ex>[l]^-{e_2} \ar@<-.5ex>[d]_-{g} \\
A \ar@<.5ex>[r]^-{f} \ar@<-.5ex>[u]_-{e_1} & B \ar@<.5ex>[l]^-{r} \ar@<-.5ex>[u]_-{s}.
}
\end{equation}
is a split epimorphism with a section $\sigma\colon{A\times_B C\to E}$ uniquely determined by $\sigma\langle 1_A,sf\rangle=e_1$ and $\sigma\langle rg,1_C\rangle=e_2$.

\item Every local product is a local coproduct.

\item Every span $(D,d,c)$ in $\C$ is equipped with a unique morphism
\begin{align*}
p_D\colon{D(d,c)\to D}
\end{align*}
 such that given any three parallel morphisms $x,y,z\colon{Z\to D}$ with $dx=dy$ and $cy=cz$ the following conditions hold \begin{align*}
dp_D\langle x,y,z\rangle=dz,\quad cp_D\langle x,y,z\rangle=cx,\\
p_D\langle x,y,y\rangle=x ,\quad p_D\langle y,y,z\rangle=z,
\end{align*}
and moreover, given any commutative diagram of the form 
\begin{align*}
\xymatrix{D_0 \ar[d]_{f_0} & D \ar[d]^{f} \ar[l]_{d} \ar[r]^{c} & D_1\ar[d]^{f_1}\\D'_0  & D' \ar[l]_{d'} \ar[r]^{c'} & D'_1}
\end{align*} 
 the induced square
\begin{align*}
\xymatrix{D(d,c)\ar[r]^-{p_D}\ar[d]_{f\times_{f_0} f\times_{f_1} f} & D\ar[d]^{f}\\
D'(d',c') \ar[r]^-{p'_{D'}} & D'}
\end{align*}
is commutative. 

In particular, given any nine parallel morphisms     $$ x_1,y_1,z_1, x_2,y_2,z_2,x_3,y_3,z_3\colon{  Z\to D
}$$
    with
    \begin{align*}
    dx_1=&dx_2  &  dx_1=&dy_1\\
    &cx_2=cx_3 && cy_1=cz_1\\
    dy_1=&dy_2  &  dx_2=&dy_2\\
    &cy_2=cy_3 && cy_2=cz_2\\
    dz_1=&dz_2  &  dx_3=&dy_3\\
    &cz_2=cz_3 && cy_3=cz_3
    \end{align*}
   then 
   \begin{align*}
      p_D\langle p_D\langle x_1,x_2,x_3\rangle,\;p_D\langle y_1,y_2,y_3\rangle,\;p_D\langle z_1,z_2,z_3\rangle\rangle=\\
      =p_D\langle p_D\langle x_1,y_1,z_1\rangle,\;p_D\langle x_2,y_2,z_2\rangle,\;p_D\langle x_3,y_3,z_3\rangle\rangle.
   \end{align*}

\end{enumerate}

\end{theorem}

\begin{remark}
The hypotheses above appear to be strictly weaker than the existence of all finite limits. Indeed, finite products already include a terminal object, and kernel pairs provide a particular class of equalizers. However, we only require pullbacks of split epimorphisms along split epimorphisms, rather than all pullbacks. We note, though, that we do not have a concrete example of a category in our context that lacks all finite limits.
\end{remark}

\section{Conclusion}

Condition (H) in Theorem~\ref{thm Main extended} shows great versatility in encoding internal categorical structures on spans or on reflexive graphs admitting a unique multiplication. It also motivates the definition of dikite diagrams and the corresponding notion of admissibility. 

Furthermore, centralization between two relations over the same base object (see, for example, \cite{MF47} and the references therein) can be expressed in this framework. Indeed, given two spans of the form $(A,\alpha,\beta f)$ and $(C,\beta g,\gamma)$, regarded as relations, they centralize each other precisely when the associated dikite diagram---such as the one displayed in Definition~\ref{def: simple}, item~(2)---is admissible. In this situation the relations need not be reflexive; rather, they may be relations in which one leg factors through a split epimorphism. This perspective is also relevant to the study of categorical ideals in the sense of~\cite{MFUvdL17}.

A further direction for research is suggested by the observation that every category possesses a kind of \emph{Mal'tsev signature}, consisting of the class of spans $(D,d,c)$ such that every dikite diagram with direction $(D,d,c)$ is admissible. With this notion, a category is naturally Mal'tsev if and only if its signature is the class of all spans; it is Mal'tsev if and only if its signature contains all relations; and it is weakly Mal'tsev if and only if its signature contains all strong relations. This perspective opens the way for a systematic study of categories through their associated Mal'tsev signatures, suggesting new avenues to understand internal structures and their interactions in categorical algebra.

\section*{Acknowledgements}

This work has previously been funded by FCT/MCTES (PIDDAC) via  the projects: 
PAMI--ROTEIRO-0328-2013 (022158);
 UIDP-04044-2020; UIDB-04044-2020;  MATIS (CENTRO-01-0145-FEDER-000014 - 3362); CENTRO-01-0247-FEDER-(069665, 039969); as well as POCI-01-0247-FEDER-(069603, 039958, 039\-863, 024533).

Furthermore, the author acknowledges the project Fruit.PV and the Shota Rustaveli National Science Foundation of Georgia (SRNSFG), through grant FR-24-9660, \emph{Categorical methods for the study of cohomology theory of monoid-like structures: an approach through Schreier extensions}.

The author also acknowledges Fundação para a Ciência e a Tecnologia (FCT) for its financial support via the project CDRSP Funding (DOI: 10.544\-99/UID/04044/2025) and ARISE funding (DOI: 10.54499/LA/P/0112/2020), and by CDRSP and ESTG from the Polytechnic of Leiria.

\end{document}